\documentclass[twoside,11pt]{entics} 
\usepackage{enticsmacro}
\usepackage{graphicx}
\usepackage[all]{xy}

\usepackage{amsmath,amssymb}
\usepackage{stmaryrd,mathpartir}
\usepackage{graphicx,graphbox,hyperref}
\usepackage{tikz-cd}
\usepackage{multicol}
\allowdisplaybreaks

\usepackage{setspace}
\newcommand{\ip}{\rotatebox[origin=c]{180}{$\pi$}}
\newcommand{\from}{\leftarrow}
\newcommand{\C}{\mathbb{C}}

\newcommand{\A}{\mathbb{A}}
\newcommand{\s}{\mathsf{succ}}
\newcommand{\z}{\mathsf{zero}}
\newcommand{\p}{\mathsf{pred}}

\makeatletter
\providecommand{\leftsquigarrow}{%
  \mathrel{\mathpalette\reflect@squig\relax}%
}
\newcommand{\reflect@squig}[2]{%
  \reflectbox{$\m@th#1\rightsquigarrow$}%
}
\makeatother

\def\conf{MFPS 2025} 	
\volume{NN}			
\def\lastname{Nester} 
\begin{document}
\begin{frontmatter}
  \title{Elgot Categories and Abacus Programs\thanksref{ALL}} 						
 \thanks[ALL]{This work was supported by the Estonian Research Council grant PRG2764, and by the European Union under Grant Agreement No. 101087529. The author would like to thank JS Lemay and Tarmo Uustalu for their interest in and encouragement of the present work, and to acknowlegde useful conversations with Filippo Bonchi in the early stages of its formation.}   
  \author{Chad Nester\thanksref{a}\thanksref{myemail}}	
  \address[a]{Institute of Computer Science\\ University of Tartu\\				
    Tartu, Estonia}  							
  \thanks[myemail]{Email: \href{mailto:nester@ut.ee} {\texttt{\normalshape
        nester@ut.ee}}} 
\begin{abstract} 
  We introduce Elgot categories, a sort of distributive monoidal category with additional structure in which the partial recursive functions are representable. Moreover, we construct an initial Elgot category, the morphisms of which coincide with a lightly modified version of Lambek's abacus programs. The partial functions that are strongly representable in this initial Elgot category are precisely the partial recursive ones.
\end{abstract}
\end{frontmatter}

\section*{Phil Scott}
This article is dedicated to the memory of Phil Scott (1947--2023).
\\\\
I was introduced to Phil in the summer of 2012 at a workshop in Halifax. At the time I was an undergraduate student, and do not remember understanding anything that took place. Nonetheless, I was made to feel welcome and my interest was warmly encouraged. A few years later, I would in turn be introduced to categorical logic by "Lambek and Scott"~\cite{Lambek1986}. Phil's book was the first material on category theory that I really took to, and it has undeniably shaped my view of the subject.

I would see Phil again from time to time, particularly during my time in Ottawa, where he lived, and in Edinburgh, where he would often visit. He was always kind, often smiling, and invariably good conversation. I think that Phil would have been interested in the work presented here, and would have very much liked to pick his brain on the subject. 

\section{Introduction}\label{sec:intro}

This paper concerns the categorical representation of the partial recursive functions. More precisely, we are interested in assumptions on distributive monoidal categories which ensure the representability of all partial recursive functions therein. We build directly on the work of Plotkin~\cite{Plotkin2013}, which supplies one such set of assumptions. Specifically, it suffices to assume that the category in question has a weak (left or right) natural numbers object $z : I \to N \from N : s$ in the sense of Par\'{e} and Rom\'{a}n~\cite{Pare1989} with the additional property that the morphism $[z,s] : I+N \to N$ is invertible, and that this inverse is a weakly final coalgebra of the functor $I + -$. 

The first contribution of this paper is another set of assumptions, of a rather different flavour, on a distributive monoidal category that ensure it represents the partial recursive functions. Specifically, we consider distributive monoidal categories whose additive monoidal structure is traced, and which additionally contain a distinguished object $N$ that is required to be isomorphic to $I+N$. We show that in any such category the conditions given by Plotkin are satisfied, meaning in particular that the partial recursive functions are representable. We propose to call categories like this in which the trace operator satisfies the uniformity axiom \emph{Elgot categories}.

The second contribution of this paper is the construction of a category whose morphisms correspond to abacus programs in the sense of Lambek~\cite{Lambek1961}. This construction is inspired by (and largely adapted from) the work of Bonchi et al.~\cite{Bonchi2023,Bonchi2024} on tape diagrams. We show that this category of abacus programs is an Elgot category, and that it is moreover an initial object in a suitable category of Elgot categories. 

We also obtain that our category of abacus programs \emph{strongly} represents the partial recursive functions. Loosely speaking, while (mere) representability is about being able to \emph{interpret} a given partial function in some category, strong representability is about being able to \emph{define} it in that category (see Section~\ref{sec:representability}). Plotkin's set of assumptions ensuring representability of the partial recursive functions can be extended to ensure their strong representability~\cite{Plotkin2013}. Specifically, it suffices to ask that the numerals $\underline{0} = z$ and $\underline{1} = zs$ are distinct, and that any partial function that is strongly representable in the category in question is necessarily partial recursive. Our category of abacus programs satisfies these extra conditions. Thus, our initial Elgot category strongly represents (so in a sense, generates) the partial recursive functions.

\subsection{Related Work}
That distributive categories model a kind of finitary branching computation has been known for some time (see e.g.,~\cite{Walters1989,Walters1992,Cockett1993,Carboni1993,Khalil1993}). However, the category of partial recursive functions is not distributive because it does not have products~\cite{Cockett2023}. Instead it has restriction products~\cite{Cockett2007}, making it a distributive restriction category in the sense of e.g.,~\cite{Cockett2023}. Every distributive restriction category is distributive monoidal, and the results of this paper can almost certainly be specialised to Elgot categories whose underlying distributive monoidal category is in fact a distributive restriction category. We leave this for future work (see Section~\ref{sec:conclusions}), and restrict our attention to the more general setting herein.

Traced monoidal categories and their string diagrams, introduced in~\cite{Joyal1996}, have become a standard tool in the categorical modelling of systems with feedback. In particular, a cocartesian monoidal category admits a parameterised iteration operator if and only if it is traced~\cite{Hasegawa1997}. This kind of iteration operator would seem to be fundamental to computer science, and has been studied extensively (see e.g., ~\cite{Elgot1975,Bloom1993,Simpson2000,Stefanescu2000,Adamek2010,Goncharov2022}). 

While distributive monoidal categories do not seem to have been much studied in and of themselves --- with the notable exception of~\cite{Labella2003} --- they are in particular rig categories\footnote{We note that what we call \emph{rig categories} are called \emph{tight bimonoidal categories} in the work of Johnson and Yau~\cite{JohnsonYau2024}.}, which are the subject of a good amount of existing work (see e.g.,~\cite{Laplaza1971,JohnsonYau2024}). Tape diagrams, which are a string-diagrammatic calculus for rig categories~\cite{Bonchi2023,Bonchi2024}, are particularly relevant to our development. The construction of the initial Elgot category of Abacus programs presented here is largely adapted from that of the rig category of tape diagrams (see Section~\ref{sec:abacus-programs} for a discussion of the differences), and the present work has been inspired in part by the evident visual similarity between abacus programs and tape diagrams.

For a historical discussion of the notion of representability see Part III of~\cite{Lambek1986}, which is itself concerned with the representability of (total) recursive funcitions in various categorical structures. The specific notion of (strong) representability used in this paper is from~\cite{Plotkin2013}, and we use the results of~\cite{Plotkin2013} repeatedly in our development. Natural numbers objects first appear in~\cite{Lawvere1964}, although the more general notion of natural numbers object in a monoidal category appears significantly later in~\cite{Pare1989}. A recent survey of these ideas and results concering them can be found in Section 4 of~\cite{Hofstra2021}.

Finally, there is a line of work in the field of reversible computation which is similar in spirit to our work. Namely, various sorts of (initial) traced rig category have been proposed as reversible programming languages~\cite{Bowman2011,James2012,Kaarsgaard2019,Chen2021,Kaarsgaard2021}.

\subsection{Synopsis}

Section~\ref{sec:preliminaries} recalls the notions of structured monoidal category that will be necessary for our development. Section~\ref{sec:representability} contains the necessary material on natural numbers algebras and (strong) representability of partial recursive functions. Section~\ref{sec:elgot-categories} contains the definition of Elgot category, and establishes that every such category represents the partial recursive functions (Theorem~\ref{thm:elgot-representable}). Section~\ref{sec:abacus-programs} introduces Lambek's abacus programs, constructs a category of abacus programs (Definition~\ref{def:abacus-programs}), establishes that it is an initial Elgot category (Theorem~\ref{thm:abacus-elgot} and Theorem~\ref{thm:abacus-initial}, and that it strongly represents all (and only) the partial recursive functions (Theorem~\ref{thm:abacus-strong-representability} and Corollary~\ref{cor:final}). The axioms for rig categories in both the usual and right-strict case are enumerated in Appendix~\ref{sec:distributive-monoidal-categories}, and Appendix~\ref{sec:proof-appendix} contains the proof of Theorem~\ref{thm:abacus-elgot}, which is long, tedious, and straightforward.
  
\section{Traced and Distributive Monoidal Categories}\label{sec:preliminaries}
In this section we set up the rest of our development by recalling the required notions of structured monoidal category: cocartesian monoidal categories (Definition~\ref{def:cocartesian-monoidal-category}), (uniform) traced monoidal categories (Definition~\ref{def:traced-monoidal-category} and Definition~\ref{def:uniform-traced-monoidal-category}), and distributive monoidal categories (Definition~\ref{def:distributive-monoidal-category}). 

We assume the reader to have some familiarity with category theory, and in particular with monoidal categories and their string diagrams. Some elementary knowledge of the partial recursive functions would also be helpful, but is not, strictly speaking, required to read this paper. We consider only strict monoidal category structure. In particular, this means that the additive and multiplicative monoidal structures of our distributive monoidal categories are both strict. Herein, ``monoidal category'' means ``strict monoidal category''. That said, one imagines that suitable analogues of the results presented here hold in the more general setting. We will write composition in diagrammatic order, meaning that given $f : A \to B$ and $g : B \to C$ we write $fg : A \to C$ and \emph{not} $gf : A \to C$. String diagrams are to be read top-to bottom.

We begin with the notion of cocartesian monoidal category:
\begin{definition}\label{def:cocartesian-monoidal-category}
  A \emph{cocartesian monoidal category} is a tuple $((\C,+,0,\sigma^+),(\mu,\eta))$ such that:
  \begin{itemize}
  \item $(\C,+,0,\sigma^+)$ is a symmetric monoidal category.
  \item $\mu$ and $\eta$ are monoidal natural transformations with components as in $\mu_A : A + A \to A$ and $\eta_A : 0 \to A$. 
  \item For all objects $A$ of $\C$, $(A,\mu_A,\eta_A)$ is a commutative monoid in $\C$. 
  \end{itemize}
\end{definition}

More explicitly, a symmetric monoidal category is cocartesian in case we have morphisms $\mu_A : A + A \to A$ and $\eta_A : 0 \to A$ for each object $A$ of $\C$ such that for all $f : A \to B$ we have:
\begin{mathpar}
  \mu_Af = (f+f)\mu_B

  \eta_Af = \eta_B
\end{mathpar}
which give that $\mu$ and $\eta$ are natural transformations of the appropriate type, along with:
\begin{mathpar}
  \mu_I = 1_I

  \mu_{A+B} = (1_A + \sigma^+_{B,A} + 1_B)(\mu_A + \mu_B)
  
  \eta_I = 1_I

  \eta_{A + B} = \eta_A + \eta_B
\end{mathpar}
which give that $\mu$ and $\eta$ are monoidal natural transformations, and finally for each $(A,\mu_A,\eta_A)$ to be a commutative monoid means that we have:
\begin{mathpar}
  (\mu_A + 1_A)\mu_A = (1_A + \mu_A)\mu_A

  \sigma^+_{A,A}\mu_A = \mu_A

  (1_A + \eta_A)\mu_A
\end{mathpar}

Morphisms of cocartesian monoidal categories may be represented string-diagrammatically by extending the diagrams for symmetric monoidal categories with the following components corresponding the $\mu$ and $\eta$:
\begin{mathpar}
    \mu
  \hspace{0.3cm}
  \leftrightsquigarrow
  \hspace{0.3cm}
  \includegraphics[height=1cm,align=c]{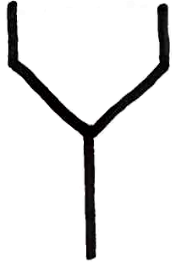}

  \eta
  \hspace{0.3cm}
  \leftrightsquigarrow
  \hspace{0.3cm}
  \includegraphics[height=1cm,align=c]{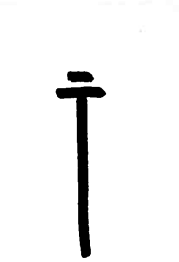}
\end{mathpar}

In a cocartesian monoidal category, $A+B$ is always a coproduct of $A$ and $B$, with the coparing $[f,g] : A+B \to C$ of $f : A \to C$ and $g : B \to C$ given by $(f+g)\mu_C$. The coprojections are given by $1_A + \eta_B : A \to A + B$ and $\eta_A + 1_B : B \to A + B$. Moreover, the unit object is necessarily initial, and so every cocartesian monoidal category has finite coproducts. In fact, asking that the tensor product and unit object of a symmetric monoidal category carry finite coproduct structure in this way is equivalent to asking that it be cocartesian monoidal~\cite{Fox1976}. 

The next notion we require is that of a traced monoidal category:
\begin{definition}\label{def:traced-monoidal-category}
  A \emph{traced monoidal category} is a tuple $((\C,\oplus,0,\sigma^\oplus),\mathsf{Tr})$ such that:
  \begin{itemize}
  \item $(\C,\oplus,0,\sigma^\oplus)$ is a symmetric monoidal category.
  \item $\mathsf{Tr}$ is a family of operations $\mathsf{Tr}^C_{A,B} : \C(A \oplus C, B \oplus C) \to \C(A,B)$ satisfying the following axioms:
  \begin{itemize}
  \item[] \textbf{[TR1]} $\mathsf{Tr}^I_{A,B}(f) = f$
  \item[] \textbf{[TR2]} $\mathsf{Tr}^{C \oplus D}_{A,B} = \mathsf{Tr}^C_{A,B}(\mathsf{Tr}^D_{A \oplus C,B \oplus C}(f)$
  \item[] \textbf{[TR3]} $g\mathsf{Tr}^C_{A,B}(f)h = \mathsf{Tr}^C_{A',B'}((g \oplus 1_C)f(h \oplus 1_C))$
  \item[] \textbf{[TR4]} $1_D \oplus \mathsf{Tr}^C_{A,B}(f) = \mathsf{Tr}^C_{D \oplus A,D \oplus B}(1_D \oplus f)$
  \item[] \textbf{[TR5]} $\mathsf{Tr}^C_{A,B}((1_A \oplus h)f) = \mathsf{Tr}^D_{A,B}(f(1_B \oplus h))$
  \item[] \textbf{[TR6]} $\mathsf{Tr}^A_{A,A}(\sigma^\oplus_{A,A}) = 1_A$
  \end{itemize}
  \end{itemize}
\end{definition}
Morphisms of traced monoidal categories admit a helpful string-diagrammatic representation as in:
\begin{mathpar}
  \mathsf{Tr}^C_{A,B}(f)
  \hspace{0.3cm}
  \leftrightsquigarrow
  \hspace{0.3cm}
  \includegraphics[height=1.4cm,align=c]{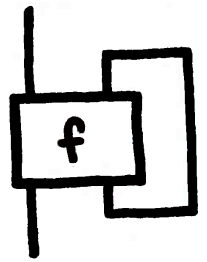}
\end{mathpar}

Some trace operators are additionally \emph{uniform}\footnote{We note that a trace operator is typically (see e.g.,~\cite{Selinger2010,Hasegawa2003}) said to be uniform with repsect to a class of morphisms $h$, called \emph{strict}, for which the uniformity axiom holds. In this paper we work with a more demanding notion of uniformity in which all morphsims are required to be strict.} in the following sense:
\begin{definition}\label{def:uniform-traced-monoidal-category}
  A traced monoidal category $((\C,\oplus,0,\sigma^\oplus),\mathsf{Tr})$ is called \emph{uniform} in case for all $f : A \oplus C \to B \oplus C$ and $g : A \oplus D \to B \oplus D$, if there exists $h : C \to D$ such that $(1_A \oplus h)f = g(1_B \oplus h)$ then we have $\mathsf{Tr}^C_{A,B}(f) = \mathsf{Tr}^D_{A,B}(g)$. 
\end{definition}

Hasegawa has shown~\cite{Hasegawa1997} that a cocartesian monoidal category is traced if and only if it admits a parameterised iteration operator in the sense of e.g., Simpson and Plotkin~\cite{Simpson2000}. Specifically, defining $(-)^\dagger : \C(A,X+A) \to \C(A,X)$ by $f^\dagger = \mathsf{Tr}^A_{A,X}(\mu_Af)$ makes $(-)^\dagger$ into a parameterised iteration operator, where the fixed point property $f^\dagger = f[1_X,f^\dagger]$ holds as in:
\[
\includegraphics[height=2.7cm,align=c]{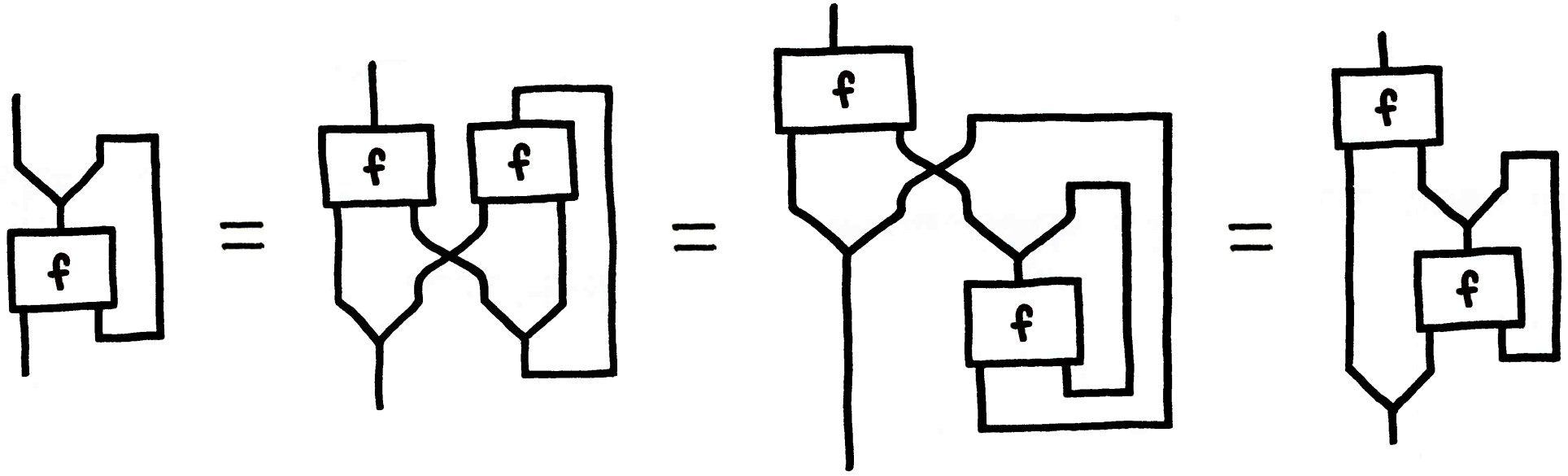}
\]

A distributive monoidal category is a category with two tensor product structures, one typically called \emph{multiplicative} and the other \emph{additive}, such that the multiplicative structure distributes over the additive one in the following sense:
\begin{definition}\label{def:distributive-monoidal-category}
  A \emph{distributive monoidal category} is a tuple $(\C,(+,0,\sigma^+),(\otimes,I),(\mu,\eta),(\lambda^\bullet,\rho^\bullet),(\delta^l,\delta^r))$ such that:
  \begin{itemize}
  \item $((\C,+,0,\sigma^+),(\mu,\eta))$ is a cocartesian monoidal category, which we refer to as the \emph{additive} monoidal structure on $\C$.
  \item $(\C,\otimes,I)$ is a monoidal category, which we refer to as the \emph{multiplicative} monoidal structure on $\C$.
  \item $\lambda^\bullet$ and $\rho^\bullet$ are natural isomorphisms with components $\lambda^\bullet_A : 0 \otimes A \to 0$ and $\rho^\bullet_A : A \otimes 0 \to 0$ called the left and right \emph{annihilator}, respectively.
  \item $\delta^l$ and $\delta^r$ are natural isomorphisms with components
    \begin{mathpar}
      \delta^l_{A,B,C} : A \otimes (B + C) \to (A \otimes B) + (A \otimes C)

      \delta^r_{A,B,C} : (A + B) \otimes C \to (A \otimes C) + (B \otimes C)
    \end{mathpar}
    called the left and right \emph{distributor}, respectively.
  \item The coherence axioms \textbf{[B1]}-\textbf{[B22]} (See Appendix~\ref{sec:distributive-monoidal-categories}) are satisfied.
  \end{itemize}
\end{definition}

We recall a common notational convenience: when referring to the objects of a distributive monoidal category one often denotes the multiplicative tensor product $\otimes$ by juxtaposition. For example, $A \otimes B$ becomes $AB$, $A \otimes B \otimes C$ becomes $ABC$, and $A \otimes (B + (C \otimes D))$ becomes $A(B + CD)$.

The notion of distributive monoidal category may be partially strictified. In particular, we say that a distributive monoidal category is \emph{right-strict} in case $\lambda^\bullet$, $\rho^\bullet$, and $\delta^r$ are identity natural transformations. The effect of these assumptions on the axioms of distributive monoidal category is treated explicitly in Appendix~\ref{sec:distributive-monoidal-categories}. There is of course a dual notion of \emph{left-strict distributive monoidal category}, although here we will work exclusively with the right-strict version.

While in this paper we privledge the point of view occupied by Definition~\ref{def:distributive-monoidal-category}, which emphasizes the rig category structure of distributive monoidal categories, it is worth mentioning that distributive monoidal categories admit a much simpler axiomatization: For any category $\C$ with monoidal structure $(\C,\otimes,I)$ and cocartesian monoidal structure $((\C,+,0,\sigma^+),(\mu,\eta))$ we may construct morphisms:
\begin{mathpar}
  \mu^l_{A,B,C} = \left[1_A \otimes \ip_0^{B,C}, 1_A \otimes \ip_1^{B,C}\right] : (A \otimes B) + (A \otimes C) \to A \otimes (B+C)

  \mu^r_{A,B,C} = \left[\ip_0^{A,B} \otimes 1_C,\ip_1^{A,B} \otimes 1_C\right] : (A \otimes C) + (B \otimes C) \to (A+B) \otimes C
\end{mathpar}
It so happens that asking for an inverse to every such map gets us most of the way to having a distributive monoidal category. Specifically, we have:
\begin{lemma}[after~\cite{Labella2003}]
  Let $((\C,+,0,\sigma^+),(\mu,\eta))$ be a cocartesian monoidal category, and let $(\C,\otimes,I)$ be a monoidal category. Suppose that $\lambda^\bullet$, $\rho^\bullet$, $\delta^l$, and $\delta^r$ are indexed families of morphisms as in Definition~\ref{def:distributive-monoidal-category}. Then $(\C,(+,0,\sigma^+),(\otimes,I),(\mu,\eta),(\lambda^\bullet,\rho^\bullet),(\delta^l,\delta^r))$ is a distributive monoidal category if and only if:
  \begin{itemize}
  \item Each $\lambda^\bullet_A : 0 \otimes A \to 0$ is the inverse of the morphism $\eta_{0 \otimes A} : 0 \to 0 \otimes A$ and, dually, each $\rho^\bullet_A : A \otimes 0 \to 0$ is the inverse of the morphism $\eta_{A \otimes 0} : 0 \to A \otimes 0$.
  \item Each $\delta^l_{A,B,C}$ is the inverse of $\mu^l_{A,B,C}$ and, dually, each $\delta^r_{A,B,C}$ is the inverse of $\mu^r_{A,B,C}$.
  \end{itemize}
\end{lemma}

\section{Categorical Representability of Partial Recursive Functions}\label{sec:representability}

In this section we review the elementary definitions concerning (strong) representability, and recall the results of Plotkin~\cite{Plotkin2013} concerning the (strong) representability of the partial recursive functions (Theorem~\ref{thm:plotkin-representable} and Theorem~\ref{thm:plotkin-strong-representable}). This section is almost entirely adapted from the work of Plotkin~\cite{Plotkin2013}, although it seems appropriate to mention the work of Lambek and Scott~\cite{Lambek1986} in connection with the categorical representability of numerical functions, and the work of Par\'{e} and Rom\'{a}n~\cite{Pare1989} in connection natural numbers objects in monoidal categories. 

We begin with the notion of natural numbers algebra, the point of which is to allow us to represent the natural numbers in a given monoidal category:
\begin{definition}
  A \emph{natural numbers algebra} in a monoidal category $(\C,\otimes,I)$ is a tuple $(N,z,s)$ such that $N$ is an object of $\C$ and $z : I \to N$, $s : N \to N$ are morphisms of $\C$. Any such natural numbers algebra defines a \emph{numeral} $\underline{n} : I \to N$ for each $n \in \mathbb{N}$, as in:
  \begin{mathpar}
    \underline{0} = z

    \underline{n+1} = \underline{n}s
  \end{mathpar}
  We will say that a natural numbers algebra is \emph{strong} in case $\underline{n} = \underline{m} \Rightarrow n = m$ for all $n,m \in \mathbb{N}$. 
\end{definition}
The notion of representability for numerical partial functions we consider is as follows: 
\begin{definition}
  Suppose $(\C,\otimes,I)$ is a monoidal category, and that $(N,z,s)$ is a natural numbers algebra therein. Given a partial function $f : \mathbb{N}^n \to \mathbb{N}$ and a morphism $\underline{f} : N^n \to N$ of $\C$, we say that \emph{$\underline{f}$ represents $f$ with respect to $(N,z,s)$} in case for all $k_1,\ldots,k_n \in \mathbb{N}$ we have:
  \begin{mathpar}
    f(k_1,\ldots,k_n) = k
    \hspace{0.3cm}
    \Rightarrow
    \hspace{0.3cm}
    (\underline{k_1} \otimes \cdots \otimes \underline{k_n})\underline{f} = \underline{k}
  \end{mathpar}
  Note that since $f$ is a partial function there may not be any such $k \in \mathbb{N}$, in which case $f(k_1,\ldots,k_n)$ is undefined and we require nothing of the morphism $(\underline{k_1} \otimes \cdots \otimes \underline{k_n})\underline{f}$. 
\end{definition}
Strong representability is now a natural strengthening of (mere) representability:
\begin{definition}
  Suppose $(\C,\otimes,I)$ is a monoidal category, and that $(N,z,s)$ is a natural numbers algebra therein. Given a partial function $f : \mathbb{N}^n \to \mathbb{N}$ and a morphism $\underline{f} : N^n \to N$ of $\C$, we say that \emph{$\underline{f}$ strongly represents $f$ with respect to $(N,z,s)$} in case for all $k_1,\ldots,k_n \in \mathbb{N}$ we have:
  \begin{mathpar}
    f(k_1,\ldots,k_n) = k
    \hspace{0.3cm}
    \Leftrightarrow
    \hspace{0.3cm}
    (\underline{k_1} \otimes \cdots \otimes \underline{k_n})\underline{f} = \underline{k}
  \end{mathpar}
  Note that in contrast to mere representability, strong representability requires that $f(k_1,\ldots,k_n)$ is undefined if and only if the morphism $(\underline{k_1} \otimes \cdots \otimes \underline{k_n})\underline{f}$ is not a numeral. That is, $(\underline{k_1} \otimes \cdots \otimes \underline{k_n})\underline{f}$ is not equal to $\underline{k}$ for any $k \in \mathbb{N}$. 
\end{definition}
One way to understand strong representability is as a more general notion of Kleene equality. For partial functions $f,g : \mathbb{N}^n \to \mathbb{N}$ recall that \emph{$f$ is Kleene Equal to $g$}, written $f \simeq g$, in case for all $k_1,\ldots,k_n \in \mathbb{N}$ we have:
\begin{mathpar}
  f(k_1,\ldots,k_n) = k
  \hspace{0.3cm}
  \Leftrightarrow
  \hspace{0.3cm}
  g(k_1,\dots,k_n) = k
\end{mathpar}
Intuitively, $f \simeq g$ says that $f$ is defined precisely when $g$ is defined, in which case they agree.

There is a monoidal category structure $(\mathsf{Par},\times,1)$ on the category $\mathsf{Par}$ of sets and partial functions where $\times$ is the cartesian product of sets and $1 = \{*\}$ is any singleton set. Let $\mathsf{succ} : \mathbb{N} \to \mathbb{N}$ and $\mathsf{zero} : 1 \to \mathbb{N}$ be the successor function and constant zero function, respectively. Then $(\mathbb{N},\mathsf{zero},\mathsf{succ})$ is a natural numbers algebra in $\mathsf{Par}$. Now, given partial funcions $f,g : \mathbb{N}^n \to \mathbb{N}$ we have that $f$ strongly represents $g$ with respect to $(\mathbb{N},\mathsf{zero},\mathsf{succ})$ if and only if $f \simeq g$. Moreover, $f$ (merely) represents $g$ with respect to $(\mathbb{N},\mathsf{zero},\mathsf{succ}$) if and only if $g$ \emph{extends} $f$.

Another way to understand strong representability is in terms of \emph{definition} of partial functions. If $(N,z,s)$ is a strong natural numbers algebra in a monoidal category $(X,\otimes,I)$ then any morphism $f : N^n \to N$ defines a partial function $\overline{f} : \mathbb{N}^n \to \mathbb{N}$ as in:
\begin{mathpar}
  \overline{f}(k_1,\ldots,k_n)
  =
  \begin{cases}
    k \text{ if } (\underline{k_1} \otimes \cdots \otimes \underline{k_n})f = \underline{k} \\
    \uparrow \text{ otherwise}
  \end{cases}
\end{mathpar}
Then $f$ strongly represents $\overline{f}$ with respect to $(N,z,s)$. Note that while any morphism strongly represents at most one partial function, a given partial function may be strongly represented by multiple morphisms.

Before moving on, let us record a convenient fact about strong natural numbers algebras:
\begin{lem}\label{lem:pred-strong-nna}
  Let $(X,\otimes,I)$ be a monoidal category, and let $(N,z,s)$ be a natural numbers algebra therein. If $\underline{0} \neq \underline{1}$ and the (total) predecessor function $\mathsf{pred} : \mathbb{N} \to \mathbb{N}$ defined as in:
  \begin{mathpar}
    \mathsf{pred}(n)
    =
    \begin{cases}
      0 \text{ if } n = 0\\
      n' \text{ if } n = n' + 1
    \end{cases}
  \end{mathpar}
  is representable with respect to $(N,z,s)$, then $(N,z,s)$ is a strong natural numbers algebra.
\end{lem}
\begin{proof}
  Let $p : N \to N$ represent $\mathsf{pred} : \mathbb{N} \to \mathbb{N}$. Suppose $\underline{n} = \underline{m}$. If $n \neq m$ then we have (without loss of generality) that $n < m$, and so $\mathsf{pred}^{m-1}(n) = 0$ and $\mathsf{pred}^{m-1}(m) = 1$. Now since $p$ represents $\mathsf{pred}$ we have $\underline{n}\,p^{m-1} = \underline{0}$ and $\underline{m}\,p^{m-1} = \underline{1}$. But then we have $\underline{0} = \underline{n}\,p^{m-1} = \underline{m}\,p^{m-1} = \underline{1}$, a contradiciton. It follows that $n = m$, as required. 
\end{proof}

For the numerals of a natural numbers algebra to behave more-or-less like the natural numbers, it must be a natural numbers object. Here we are interested in the ``weak'' version of this notion:
\begin{definition}\label{def:weak-left-nno}
  A \emph{weak left\footnote{There is of course a dual notion of \emph{weak right natural numbers object} in which $h : B \otimes N \to A$ instead. In a symmetric monoidal category the two notions coincide, and in our development anything true of one will be true of the other. We work with the ``left'' version, but could easily have chosen to work with the ``right'' one.} natural numbers object} in a monoidal category $(\C,\otimes,I)$ is a natural numbers algebra $(N,z,s)$ in $(\C,\otimes,I)$  with the property that for any diagram $B \stackrel{b}{\to} A \stackrel{a}{\from} A$ in $\C$ there exists a morphism $h : N \otimes B \to A$ of $\C$ such that
  \[
  \begin{tikzcd}
    B \ar[d,"1_B"'] \ar[r,"z \otimes 1_B"] & N \otimes B \ar[d,"h"] & N \otimes B \ar[l,"s \otimes 1_B"'] \ar[d,"h"] \\
    B \ar[r,"b"'] & A & \ar[l,"a"] A
  \end{tikzcd}
  \]
  Note in particular that $h$ need not be the unique such morphism.
\end{definition}
Weak natural numbers objects relate to representability of numerical partial functions as follows:
\begin{prop}[Plotkin~\cite{Plotkin2013}]
  Let $(\C,\otimes,I)$ be a monoidal category, and let $(N,z,s)$ be a weak left natural numbers object in $\C$. Then all primitive recursive functions are representable in $\C$ relative to $(N,z,s)$.
\end{prop}

We will be particularly interested in the case where the ambient monoidal category $(\C,\otimes,I)$ is the multiplicative fragment of a distributive monoidal category. Then a natural numbers algebra is equivalently a morphism $\iota : I + N \to N$. From any such morphism we recover a natural numbers algebra $(N,\ip_0\iota,\ip_1\iota)$ and conversely from any natural numbers algebra $(N,z,s)$ we obtain $\iota = [z,s] : I+N \to N$. Moreover, these two constructions are inverses. Thus, a natural numbers algebra in given distributive monoidal category is equivalently a tuple $(N,\iota)$ where $\iota : I + N \to N$. 

We give an equivalent formulation of Definition~\ref{def:weak-left-nno} in this new setting:
\begin{lem}\label{lem:alternative-form-nno}
  In a distributive monoidal category, a natural numbers algebra $(N,\iota)$ defines a weak left natural numbers object $(N,z,s)$ with $[z,s] = \iota$ if and only if for any diagram $B \stackrel{b}{\to} A \stackrel{a}{\from} A$ there exists a morphism $h : NB \to A$ such that:
  \begin{mathpar}
    \begin{tikzcd}
        (I+N)B \ar[r,"\iota \otimes 1_B"] \ar[d,"\delta^r_{I,N,B}"'] & NB \ar[d,"h"] \\
        B + NB \ar[r,"{[b,ha]}"'] & A
      \end{tikzcd}
  \end{mathpar}
\end{lem}
\begin{proof}
  Suppose that $(N,z,s)$ is a weak left natural numbers object. Then we have:
  \begin{align*}
    & \delta^r_{I,N,B}[b,ha]
    = \delta^r_{I,N,B}[(z \otimes 1_B)h,(s \otimes 1_B)h]
    = \delta^r_{I,N,B}((z \otimes 1_B) + (s \otimes 1_B))\mu_{NB}h
    \\&= ((z + s) \otimes 1_B)\delta^r_{N,N,B}\mu_{NB}h
    = ((z+s) \otimes 1_B)(\mu_N \otimes 1_B)h
    = ([z,s] \otimes 1_B)h
    = (\iota \otimes 1_B)h
  \end{align*}
  Conversely, suppose that we have $\delta^r_{I,N,B}[b,ha] = (\iota \otimes 1_B)h$. Then we have:
  \begin{align*}
    & (z \otimes 1_B)h
    = (\ip_0\iota \otimes 1_B)h
    = (\ip_0 \otimes 1_B)(\iota \otimes 1_B)h
    = (\ip_0 \otimes 1_B)\delta^r_{I,N,B}[b,ha]
    = \ip_0[b,ha]
    = b
  \end{align*}
  and
  \begin{align*}
    & (s \otimes 1_B)h
    = (\ip_1\iota \otimes 1_B)h
    = (\ip_1 \otimes 1_B)(\iota \otimes 1_B)h
    = (\ip_1 \otimes 1_B)\delta^r_{I,N,B}[b,ha]
    = \ip_1[b,ha]
    = ha
  \end{align*}
  The claim follows. 
\end{proof}

From this new perspective, natural numbers algebras are algebras of the $I+-$ functor. We will also be interested in coalgebras of this functor, which we name appropriately:
\begin{definition}
  A \emph{natural numbers coalgebra} in a distributive monoidal category is a tuple $(N,\gamma)$ where $\gamma : N \to I+N$. A natural numbers coalgebra $(N,\gamma)$ in a distributive monoidal category is said to be \emph{weakly final} in case for all $\beta : A \to I + A$ there exists a morphism $h : A \to N$ such that
  \[
  \begin{tikzcd}
    A \ar[d,"h"'] \ar[r,"\beta"] & I + A \ar[d,"1_I + h"] \\
    N \ar[r,"\gamma"'] & I + N
  \end{tikzcd}
  \]
  Note that $h$ need not be the unique such morphism. 
\end{definition}
Of particular interest is the case in which we have an isomorphism $I+N \cong N$. That is, a natural numbers algebra $\iota : I+N \to N$ and a natural numbers coalgebra $\iota^{-1} : N \to I+N$ which are mutually inverse. In this case, the numerals defined by $\iota$ are augmented by a sort of predecessor operation, given by $\iota^{-1}$. Moreover, straightforward conditions on $\iota$ and $\iota^{-1}$ yield representability of all partial recursive functions: 
\begin{theorem}[after Plotkin~\cite{Plotkin2013}]\label{thm:plotkin-representable}
  In a distributive monoidal category $\C$, suppose that:
  \begin{itemize}
  \item $\iota : I+N \to N$ is an isomorphism with inverse $\iota^{-1} : N \to I+N$.
  \item $(N,\iota)$ is a weak left natural numbers object.
  \item $(N,\iota^{-1})$ is a weakly final natural numbers coalgebra.
  \end{itemize}
  Then all partial recursive functions are representable in $\C$ relative to $(N,z,s)$ where $[z,s] = \iota$. 
\end{theorem}

There is a version of Theorem~\ref{thm:plotkin-representable} for strong representability:
\begin{theorem}[after Plotkin~\cite{Plotkin2013}]\label{thm:plotkin-strong-representable}
  In a distributive monoidal category $\C$, suppose that:
  \begin{itemize}
  \item $\iota : I+N \to N$ is an isomorphism with inverse $\iota^{-1} : N \to I+N$.
  \item $(N,\iota)$ is a weak left natural numbers object.
  \item $(N,\iota^{-1})$ is a weakly final natural numbers coalgebra.
  \item $z \neq zs : I \to N$ where $[z,s] = \iota$.
  \item Every partial function that is strongly representable in $\C$ is partial recursive. 
  \end{itemize}
  Then all partial recursive functions are strongly representable in $\C$ relative to $(N,z,s)$ where $[z,s] = \iota$. 
\end{theorem}
It is worth noting that Theorem~\ref{thm:plotkin-representable} and Theorem~\ref{thm:plotkin-strong-representable} are in fact weaker, more specific versions of the results obtained by Plotkin in~\cite{Plotkin2013}, which assume slightly less structure. We have chosen to present the weaker statements here in order to more easily integrate these results into our development.

In the situation described by Theorem~\ref{thm:plotkin-strong-representable} the natural numbers algebra $(N,\iota)$ is always strong:
\begin{lemma}
  Let $(N,\iota)$ be a natural numbers algebra in a distributive monoidal category and suppose that $\iota$ in invertible. If $\underline{0} \neq \underline{1}$ then $(N,\iota)$ is a strong natural numbers algebra. 
\end{lemma}
\begin{proof}
  Lemma~\ref{lem:pred-strong-nna} tells us that it suffices to show that $\mathsf{pred} : \mathbb{N} \to \mathbb{N}$ is representable. Define $p = \iota^{-1}[z,1_N] : N \to N$ and let $[z,s] = \iota$. Then we have $zp = \ip_0\iota\iota^{-1}[z,1_N] = z$ and $sp = \ip_1\iota\iota^{-1}[z,1_N] = 1_N$. It follows that $p$ represents $\mathsf{pred}$, and so $(N,\iota)$ is strong.
\end{proof}

\section{Elgot Categories}\label{sec:elgot-categories}

In this section we introduce (pre-)Elgot categories, and show that in any (pre-)Elgot category the partial recursive functions are representable (Theorem~\ref{thm:elgot-representable}). We do this by showing that any pre-Elgot category satisfies the conditions of Theorem~\ref{thm:plotkin-representable}. Of particular interest is the way in which the traced cocartesian monoidal structure gives rise to the necessary weakly initial and weakly final algebra structure (Lemma~\ref{lem:pre-elgot-representable}).

We begin with the definition of (pre-)Elgot category:
\begin{definition}
  A \emph{pre-Elgot category} is a tuple $(\C,(+,0,\sigma^+),(\otimes,I),(\mu,\eta),\mathsf{Tr},(\lambda^\bullet,\rho^\bullet),(\delta^l,\delta^r),(N,\iota))$ such that:
  \begin{enumerate}
  \item $(\C,(+,0,\sigma^+),(\otimes,I),(\mu,\eta),(\lambda^\bullet,\rho^\bullet),(\delta^l,\delta^r))$ is a distributive monoidal category.
  \item $((\C,+,0,\sigma^+),\mathsf{Tr})$ is a traced monoidal category.
  \item $N$ is an object of $\C$ with $\iota : I+N \to N$ an isomorphism.
  \end{enumerate}
  An \emph{Elgot category} is a pre-elgot category in which the trace is uniform. Moreover, a given (pre-)Elgot category is \emph{right-strict} in case the underlying distributive monoidal category structure is right-strict.
\end{definition}

While the results of this section hold for any pre-Elgot category, the construction of an initial Elgot category of abacus programs given in Section~\ref{sec:abacus-programs} requires a uniform trace operator. Our naming scheme is motivated by the convention that an ``Elgot iteration operator'' is implied to satisfy the uniformity axiom. Examples of Elgot categories (ignoring differences in strictness for the moment) include the category of sets and partial functions, and the category of sets and partial recursive functions.

In any pre-Elgot category the conditions of Theorem~\ref{thm:plotkin-representable} are satisfied:
\begin{lemma}\label{lem:pre-elgot-representable}
  For any pre-Elgot category $(\C,(+,0,\sigma^+),(\otimes,I),(\mu,\eta),\mathsf{Tr},(\lambda^\bullet,\rho^\bullet),(\delta^l,\delta^r),(N,\iota))$ we have:
  \begin{enumerate}
  \item $(N,\iota)$ is a weak left natural numbers object in $\C$.
  \item $(N,\iota^{-1})$ is a weakly final natural numbers coalgebra in $\C$. 
  \end{enumerate}
\end{lemma}
\begin{proof}
\begin{enumerate}
\item We will use the alternate presentation given in Lemma~\ref{lem:alternative-form-nno}. Suppose we have $B \stackrel{b}{\to} A \stackrel{a}{\from} A$. Let $h : NB \to A$ be given as in:
    \[
    h
    =
    \mathsf{Tr}^{NA}_{NB,A}([1_N \otimes b,1_N \otimes a](\iota^{-1} \otimes 1_A)\delta^r_{I,N,A})
    \]
    Then we have:
    \begin{align*}
      & (\iota \otimes 1_B)h
      = (\iota \otimes 1_B)\mathsf{Tr}^{NA}_{NB,A}([1_N \otimes b,1_N \otimes a](\iota^{-1} \otimes 1_A)\delta^r_{I,N,A}) 
      \\&\stackrel{*}{=} \delta^r_{I,N,B}[b,\mathsf{Tr}^{NA}_{NB,A}([1_N \otimes b,1_N \otimes a](\iota^{-1} \otimes 1_A)\delta^r_{I,N,A})a]
      = \delta^r_{I,N,B}[b,ha]
    \end{align*}
    where the marked equation ($\stackrel{*}{=}$) holds as in:
    \begin{mathpar}
      \includegraphics[height=5cm,align=c]{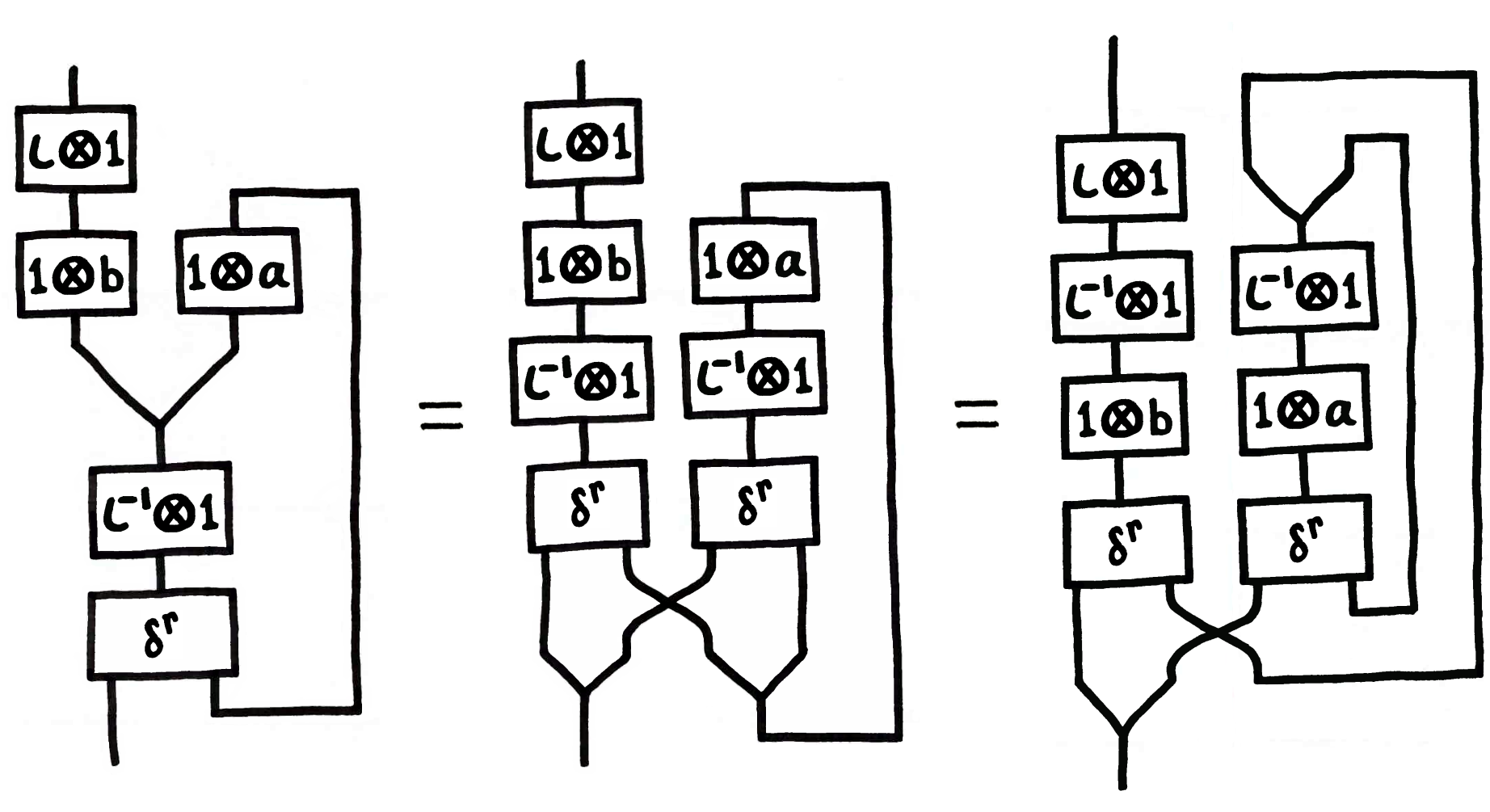}
    \end{mathpar}
    \begin{mathpar}
      \includegraphics[height=5cm,align=c]{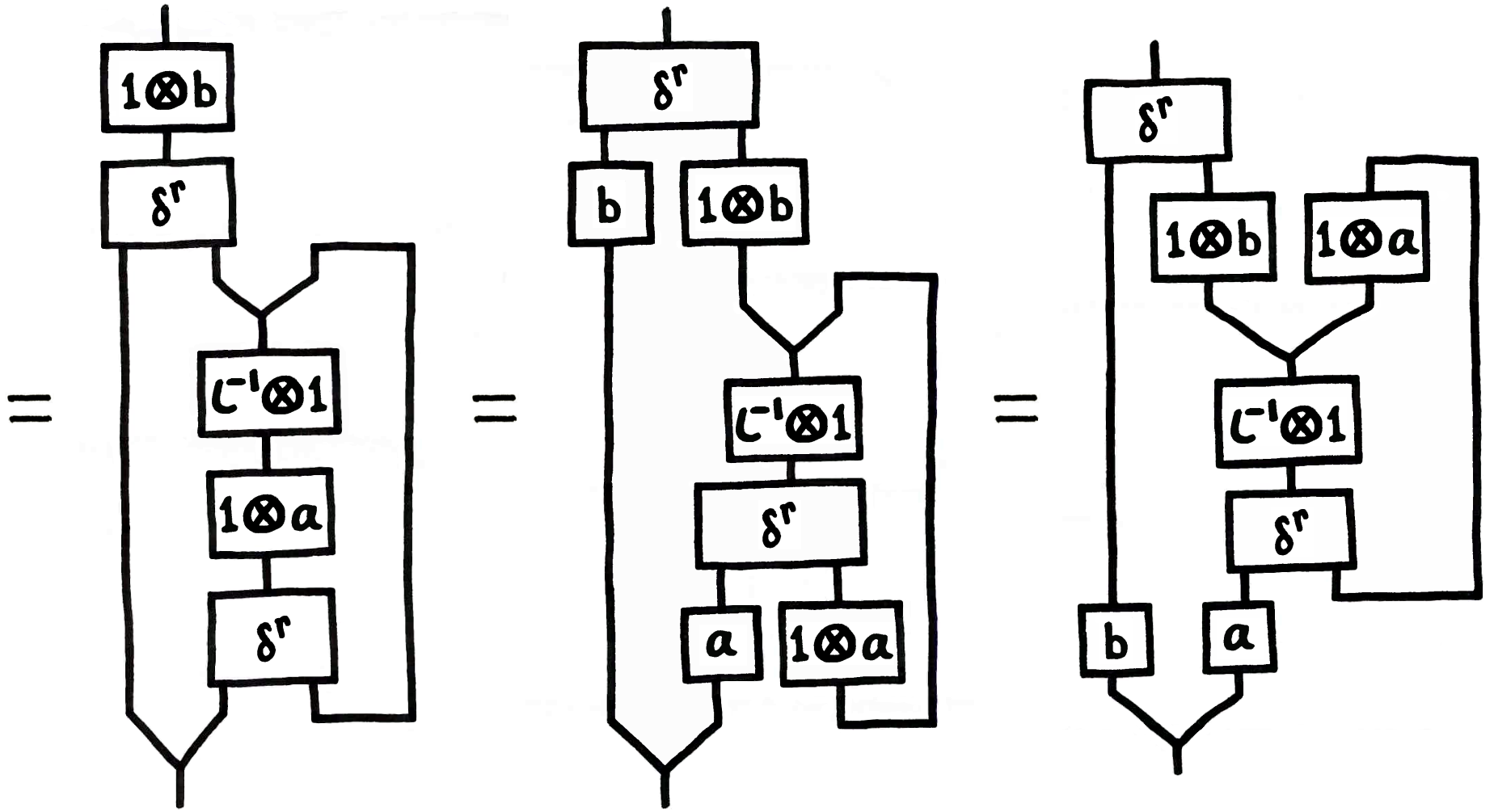}
    \end{mathpar}
    The claim follows.
\item We must show that for any $\beta : A \to I+A$ there exists $h : A \to N$ such that:
    \begin{mathpar}
      \begin{tikzcd}
        A \ar[d,"h"'] \ar[r,"\beta"] & I+A \ar[d,"1_I + h"] \\
        N \ar[r,"{\iota^{-1}}"'] & I+N 
      \end{tikzcd}
      
      \text{or equivalently}

      \begin{tikzcd}
        A \ar[d,"h"'] \ar[r,"\beta"] & I+A \ar[d,"1_I + h"]  \\
        N & I+N \ar[l,"\iota"] 
      \end{tikzcd}
    \end{mathpar}
    To that end, suppose $\beta : A \to I + A$. Let $\iota = [z,s]$, and define $h : A \to N$ as in:
    \[
    h
    =
    \mathsf{Tr}^{AN}_{A,N}([1_A \otimes z, 1_A \otimes s](\beta \otimes 1_N)\delta^r_{I,A,N})
    \]
    Then we have:
    \begin{align*}
      & h
      = \mathsf{Tr}^{AN}_{A,N}([1_A \otimes z, 1_A \otimes s](\beta \otimes 1_N)\delta^r_{I,A,N})
      \\&\stackrel{*}{=} \beta(1_I + \mathsf{Tr}^{AN}_{A,N}([1_A \otimes z, 1_A \otimes s](\beta \otimes 1_N)\delta^r_{I,A,N}))\iota
      = \beta(1_I + h)\iota
    \end{align*}
  where the marked equation ($\stackrel{*}{=}$) holds as in:
\begin{mathpar}
  \includegraphics[height=4.5cm,align=c]{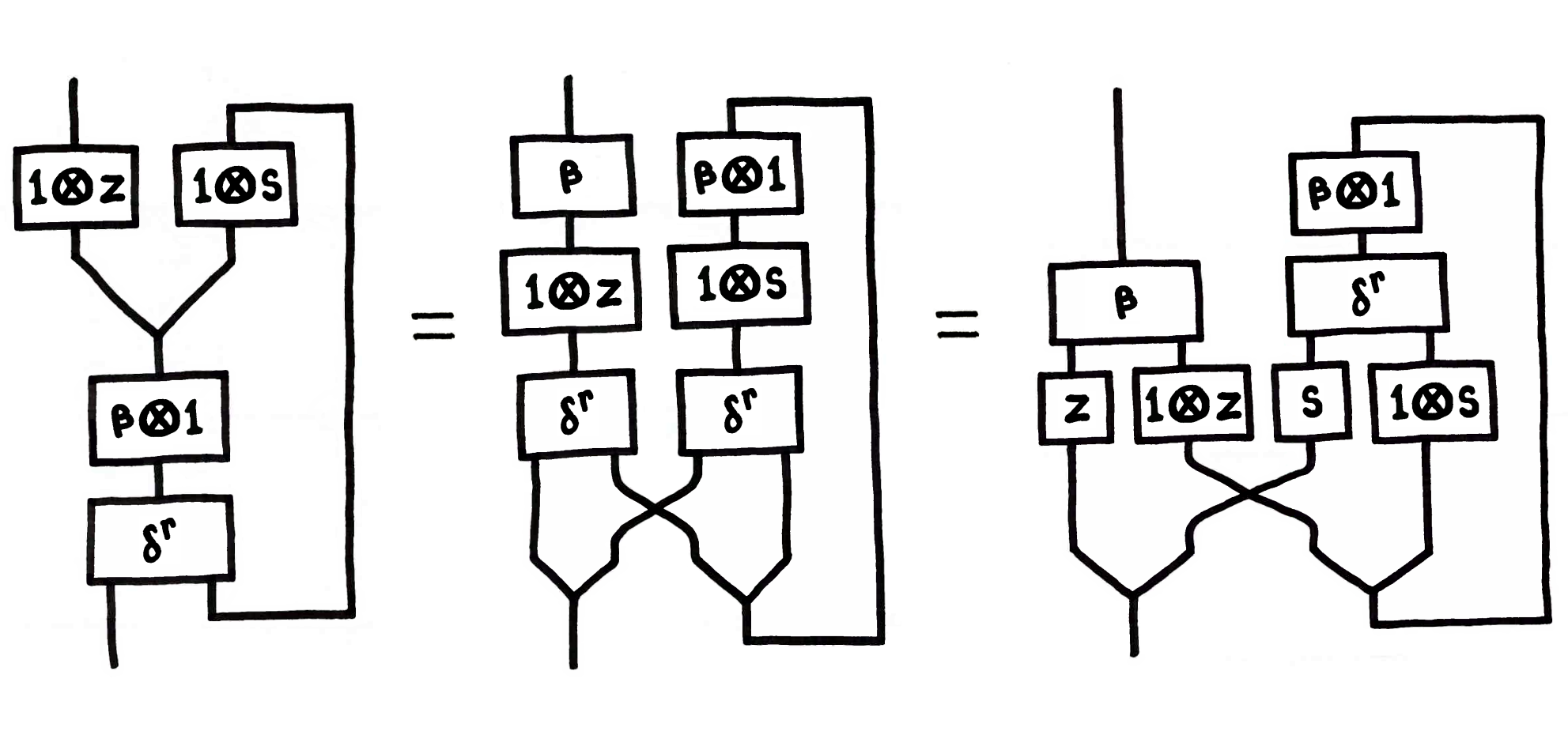}
\end{mathpar}
\begin{mathpar}
  \includegraphics[height=4.5cm,align=c]{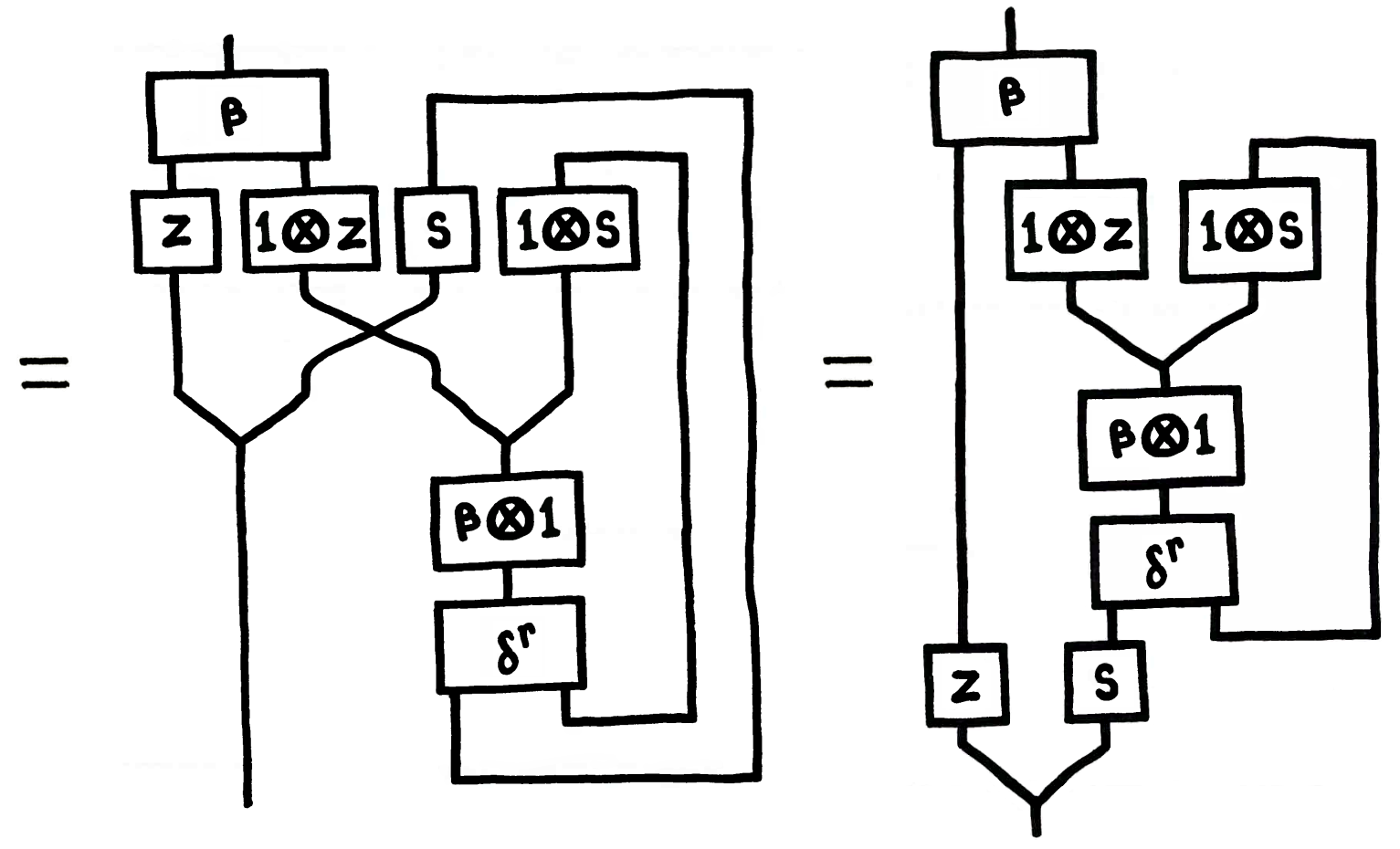}
\end{mathpar}
The claim follows.
\end{enumerate}
\end{proof}

Now as a corollary we obtain our representability result:
\begin{theorem}\label{thm:elgot-representable}
  Every partial recursive function is representable in any pre-Elgot category.
\end{theorem}
\begin{proof}
Immediate from Lemma~\ref{lem:pre-elgot-representable} and Theorem~\ref{thm:plotkin-representable}. 
\end{proof}

\section{Abacus Programs}\label{sec:abacus-programs}
In this section we introduce Lambek's abacus programs~\cite{Lambek1961}, construct a category of abacus programs (Definition~\ref{def:abacus-programs}), show that it is an Elgot category (Theorem~\ref{thm:abacus-elgot}), that it strongly represents the partial recursive functions (Theorem~\ref{thm:abacus-strong-representability}), and that it is an initial object in a category of right-strict Elgot categories and structure-preserving functors (Theorem~\ref{thm:abacus-initial}). 

What we call \emph{abacus programs} are a Turing-complete model of computation introduced by Lambek~\cite{Lambek1961}, in which one assumes a countably infinite set of \emph{locations}, each of which contains a natural number of \emph{counters}. Programs are specified as flowcharts with two sorts of instruction for each location $X$:
\begin{mathpar}
  \includegraphics[height=1.7cm,align=c]{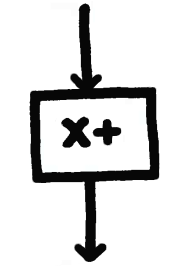}

  \includegraphics[height=1.7cm,align=c]{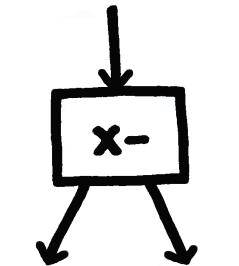}
\end{mathpar}
The instruction above left adds a single counter to location $X$, following the only arrow out of the instruction box to reach the next instruction. The effect of the instruction above right depends on the number of counters $X$ contains. If it is zero, then it remains zero and we follow the left arrow to the next instruction box. If it is nonzero, then we remove a counter from $X$ and follow the right arrow instead\footnote{Lambek's original presentation uses the opposite convention.}. For example, the following abacus program moves all of the counters from location $X$ to location $Y$:
\begin{mathpar}
  \includegraphics[height=3cm,align=c]{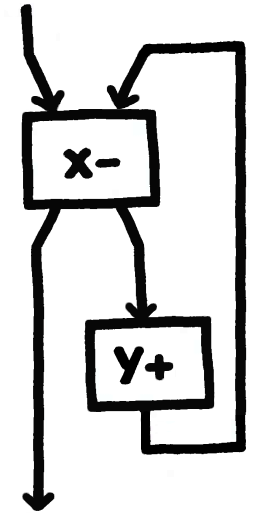}
\end{mathpar}
Arrows with no source are \emph{entry points} to the program, and arrows with no target are \emph{exit points}. The textbook of Boolos, Burgess, and Jeffery treats abacus programs at some length~\cite[Chapter~5]{Boolos2002}, and is highly recommended to the reader interested in a more detailed exposition.


We next turn to the construction of a category of abacus programs. Before we can proceed, we require some notation and terminology relating to the free monoid $X^*$ on a set $X$, and also to the the free monoid $(X^*)^*$ on $X^*$. Elements of $X^*$, which we will refer to as \emph{monomials over $X$}, are sequences of elements of $X$, and the monoid operation is concatenation of sequences. We write elements of $X^*$ by juxtaposing the necessary elements of $X$, as in $abc \in X^*$ for $a,b,c \in X$. We write $I$ for the empty such sequence. Note in particular that since $I$ is the unit of the monoid structure on $X^*$ that we have $IU = U = UI$ for any $U \in X^*$. The letters $U,V,W,\ldots$ will tend to denote monomials.

Elements of $(X^*)^*$, which we will refer to as \emph{polynomials over $X$}, are of course sequences of monomials with the monoid operation given by concatenation. We write elements of $(X^*)^*$ as sequences delimited by the symbol '$+$', as in $U + V + W \in (X^*)^*$ for $U,V,W \in X^*$. We write $0$ for the empty such sequence. Note in particular that since $0$ is the unit of the monoid structure on $(X^*)^*$ we have $0 + P = P = P + 0$ for any $P \in (X^*)^*$. The letters $P,Q,R,\ldots$ will tend to denote polynomials.

For example, the following are all monomials over $\{A,B\}$:
\begin{mathpar}
  I

  A

  AB

  BBB

  BAABBBA
\end{mathpar}
and the following are all polynomials over $\{A,B\}$ that are not monomials:
\begin{mathpar}
  0

  BA + A

  AAB + BBBBB

  A + A + 0
\end{mathpar}

We proceed to construct a category of abacus programs:
\begin{definition}\label{def:abacus-programs}
  Let $\A$ be the category defined as follows: objects of $\A$ are polynomials over a single generator $N$, so that $\A_0 = (\{N\}^*)^*$. Arrows of $\A$ are generated according to the following inference rules:
  \begin{mathparpagebreakable}
    \inferrule{U,V \text{ monomial}}
              {\s_{U,V} : UNV \to UNV}

    \inferrule{U,V, \text{ monomial}}
              {\z_{U,V} : UV \to UNV}

    \inferrule{U,V \text{ monomial}}
              {\p_{U,V} : UNV \to UV + UNV}

    \inferrule{U \text{ monomial}}             
              {1_U : U \to U}

    \inferrule{U \text{ monomial}}
              {\eta_U : 0 \to U}

    \inferrule{U \text{ monomial}}
              {\mu_U : U + U \to U}

    \inferrule{\text{}}
              {1_0 : 0 \to 0}

    \inferrule{U,W \text{ monomial}}
              {\sigma^+_{U,W} : U + W \to W + U}

    \inferrule{f : P \to Q \\ g : Q \to R}
              {fg : P \to R}

    \inferrule{f : P \to Q \\ g : R \to S}              
              {f + g : P + R \to Q + S}

    \inferrule{W \text{ monomial} \\ P,Q \text{ polynomial} \\ f : P + W \to Q + W}
              {\mathsf{Tr}^W_{P,Q}(f) : P \to Q}
  \end{mathparpagebreakable}
  These arrows are subject to the axioms of a traced distributive monoidal category. We give these explicitly in stages. At each stage, we will inductively construct the relevant structural morphisms, and then give the corresponding equations.

  First, we ensure that $(\A,+,0)$ is a monoidal category. This requires an identity morphism $1_P : P \to P$ for any polynomial $P$, which we define inductively. Any such polynomial $P$ is either $0$, in which case $1_0$ exists by assumption, or $U+Q$ for some monomial $U$ and polynomial $Q$, in which case we define $1_{U+Q} = 1_U + 1_Q$, noting that $1_U$ is assumed to exist for all monomials $U$. Now the following equations ensure that $(\A,+,0)$ is in fact monoidal:
  \begin{mathpar}
    \textbf{[S1]} (fg)h = f(gh)
    
    \textbf{[S2]} 1_P f = f = f 1_Q
    \\
    \textbf{[S3]} (f + g)(h + k) = fh + gk
    
    \textbf{[S4]} 1_0 + f = f = f + 1_0
    
    \textbf{[S5]} (f + g) + h = f + (g + h)
  \end{mathpar}

  Next, we must ensure that $(\A,+,0,\sigma^+)$ is a symmetric monoidal category. For any monomial $U$ we define $\sigma^+_{U,Q}$ for an arbitrary polynomial $Q$ by structural induction on $Q$ as in:
  \begin{mathpar}
    \sigma^+_{U,0} = 1_U
    
    \sigma^+_{U,V+Q} = (\sigma^+_{U,V} + 1_Q)(1_V + \sigma^+_{U,Q})
  \end{mathpar}
  and now for any polynomials $P,Q$ we define $\sigma^+_{P,Q}$ by structural induction on $P$ as in:
  \begin{mathpar}
    \sigma^+_{0,Q} = 1_Q
    
    \sigma^+_{U+P,Q} = (1_U + \sigma^+_{P,Q})(\sigma^+_{U,Q} + 1_P)
  \end{mathpar}
  and we may now state the corresponding equations:
  \begin{mathpar}
    \textbf{[S6]} \sigma^+_{P,Q}\sigma^+_{Q,P} = 1_{P+Q}

    \textbf{[S7]} (f + 1_Q)\sigma^+_{P',Q} = \sigma^+_{P,Q}(1_Q + f)
  \end{mathpar}

  For the cocartesian structure, we define $\eta_P$ and $\mu_P$ for arbitrary polynomials $P$ by structural induction on $P$, as in:
  \begin{mathpar}
    \eta_0 = 1_0
          
    \eta_{U+P} = \eta_U + \eta_P
    \\
    \mu_0 = 1_0
    
    \mu_{U+P} = (1_U + \sigma^+_{P,U} + 1_P)(\mu_U + \mu_P)
  \end{mathpar}
  and the corresponding equations are as follows:
  \begin{mathpar}
    \textbf{[S8]} (\mu_P + 1_P)\mu_P = (1_P + \mu_P)\mu_P

    \textbf{[S9]} \sigma^+_{P,P}\mu_P = \mu_P

    \textbf{[S10]} (\eta_P + 1_P)\mu_P = 1_P

    \textbf{[S11]} (f + f)\mu_{Q} = \mu_Pf

    \textbf{[S12]} \eta_Pf = \eta_{Q}
  \end{mathpar}

  Finally, for any polynomials $P,Q,R$ and $f : P+R \to Q+R$ we define $\mathsf{Tr}^R_{P,Q}$ by structural induction on $R$, as in:
  \begin{mathpar}
    \mathsf{Tr}^0_{P,Q}(f) = f

      \mathsf{Tr}^{W+R}_{P,Q}(f) = \mathsf{Tr}^{W}_{P,Q}(\mathsf{Tr}^{R}_{P+W,Q+W}(f))
  \end{mathpar}
  with the corresponding equations being:
  \begin{mathpar}
    \textbf{[S13]} g\mathsf{Tr}^R_{P,Q}(f)h = \mathsf{Tr}^R_{P',Q'}((g + 1_R)f(h+1_R))
    
    \textbf{[S14]} 1_S + \mathsf{Tr}^R_{P,Q}(f) = \mathsf{Tr}^R_{S+P,S+Q}(1_S + f)
    
    \textbf{[S15]} \mathsf{Tr}^R_{P,Q}((1_P + h)f) = \mathsf{Tr}^{R'}_{P,Q}(f(1_Q + h))
    
    \textbf{[S16]} \mathsf{Tr}^P_{P,P}(\sigma^+_{P,P}) = 1_P
  \end{mathpar}

  Further, morphisms of $\A$ are subject to a number of equations concerning $\mathsf{succ}$, $\mathsf{zero}$, and $\mathsf{pred}$:
  \begin{itemize}
  \item[] \textbf{[N1]} $(\z_{U,V} + \s_{U,V})\mu_{UNV}\p_{U,V} = 1_{UV+UNV}$

  \item[] \textbf{[N2]} $\p_{U,V}(\z_{U,V} + \s_{U,V})\mu_{UNV} = 1_{UNV}$

  \item[] \textbf{[N3]} $\s_{U,VNW} \s_{UNV,W} = \s_{UNV,W} \s_{U,VNW}$

  \item[] \textbf{[N4]} $\s_{U,VW} \z_{UNV,W} = \z_{UNV,W} \s_{U,VNW}$
    
  \item[] \textbf{[N5]} $\z_{U,VNW} \s_{UNV,W} = \s_{UV,W} \z_{U,VNW}$

  \item[] \textbf{[N6]} $\z_{U,VW} \z_{UNV,W} = \z_{UV,W} \z_{U,VNW}$

  \item[] \textbf{[N7]} $\s_{U,VNW} \p_{UNV,W} = \p_{UNV,W} (\s_{U,VW} + \s_{U,VNW})$

  \item[] \textbf{[N8]} $\z_{U,VNW} \p_{UNV,W} = \p_{UV,W} (\z_{U,VW} + \z_{U,VNW})$

  \item[] \textbf{[N9]} $\s_{UNV,W} \p_{U,VNW} = \p_{U,VNW}(\s_{UV,W} + \s_{UNV,W})$

  \item[] \textbf{[N10]} $\z_{UNV,W} \p_{U,VNW} = \p_{U,VW}(\z_{UV,W} + \z_{UNV,W})$

  \item[] \textbf{[N11]} $\p_{U,VNW}(\p_{UV,W} + \p_{UNV,W}) \\= \p_{UNV,W}(\p_{U,VW} + \p_{U,VNW})(1_{UVW} + \sigma^+_{UNVW,UVNW} + 1_{UNVNW})$
  \end{itemize}
\end{definition}

Morphisms of $\A$ can be understood as a modified notion of abacus program in which the locations used by the program are tracked explicitly. This modified notion of abacus program has the same expressive power as the original one. Briefly, the arrows of our flowcharts become wires in string diagrams corresponding to the additive structure of $\A$, and every such wire is now labelled with the non-repeating list of location names used by the program at that point. Explicitly, there are now three sorts of instruction:
\begin{mathpar}
  \z_{U,V}
  \hspace{0.3cm}
  \leftrightsquigarrow
  \hspace{0.3cm}
  \includegraphics[height=1.7cm,align=c]{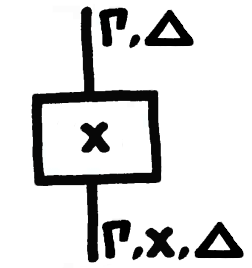}
  
  \s_{U,V}
  \hspace{0.3cm}
  \leftrightsquigarrow
  \hspace{0.3cm}
  \includegraphics[height=1.7cm,align=c]{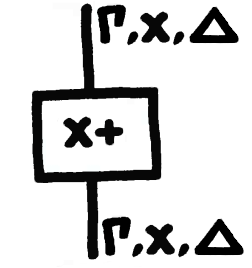}

  \p_{U,V}
  \hspace{0.3cm}
  \leftrightsquigarrow
  \hspace{0.3cm}
  \includegraphics[height=1.7cm,align=c]{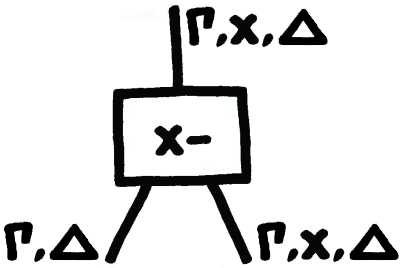}
\end{mathpar}
The instructions above middle and above right work as before, although note that the $X-$ instruction now removes the location $X$ from the list in the case where it no longer contains counters. The instruction above left is new, and adds a new location $X$ containing no counters to the list. We reiterate that in these diagrams the lists of locations are assumed to be non-repeating, which in particular means that whenever we write $\Gamma,X,\Delta$ it is implied that neither $\Gamma$ nor $\Delta$ contains $X$. These non-repeating lists correspond to the monomial part of objects in $\A$ with, for example, the list $X$,$Y$,$Z$,$W$ corresponding to the monomial $NNNN$. The other differences from unmodified abacus programs are superficial: instead of directed arrows we draw wires, adopting the convention that a wire connected to the top of an instruction box is ``incoming'' and an arrow connected to the bottom is ``outgoing''. Instruction boxes with multiple in-degree must now be constructed using the cocartesian monoidal structure, and arrows that point ``backwards'' must be constructed using the trace operator.

The location $X$ labelling each instruction serves only to indicate a position in the relevant list, and in $\A$ this positional information is present in the monomial subscripts of $\z$,$\s$, and $\p$. For example, the morphism $\s_{NN,N}$ of $\A$ corresponds to a modified abacus program concerned with a sequence of four locations that has the effect of adding a new counter to the third location in the sequence. This more structural representation means that modified abacus programs which are identical in all  but choice of location names correspond to the same morphism of $\A$. For example, recall the abacus program that moves all of the counters from location $X$ to location $Y$. From this one obtains a modified abacus program concerning only locations $X$ and $Y$ and the corresponding morphism of $\A$ as in:
\begin{mathpar}
  \mathsf{Tr}^{NN}_{NN,N}(\mu_{NN} \p_{I,N} (1_N + \s_{N,I}))
  \hspace{0.5cm}
  \leftrightsquigarrow
  \hspace{0.5cm}
  \includegraphics[height=2.8cm,align=c]{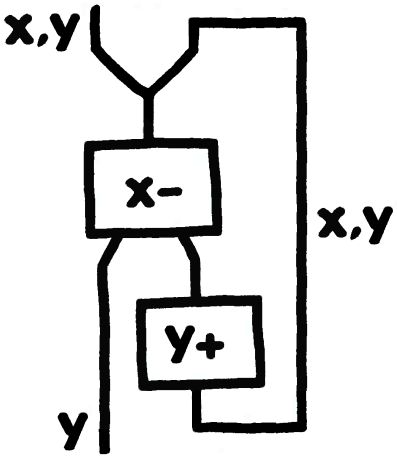}  
\end{mathpar}

The axioms \textbf{[N1-N11]} of $\A$ are more easily understood by considering the corresponding abacus program. For example, from this perspective the axiom \textbf{[N3]} states that it does not matter which order one adds counters to distinct locations in:
\begin{mathpar}
  \includegraphics[height=2.4cm,align=c]{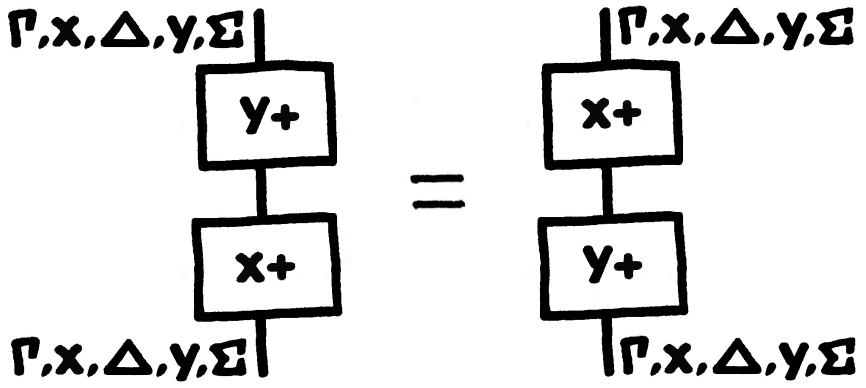}
\end{mathpar}
and the axiom \textbf{[N11]} states something similar, if more involved, about the predecessor operation:
\begin{mathpar}
  \includegraphics[height=3.5cm,align=c]{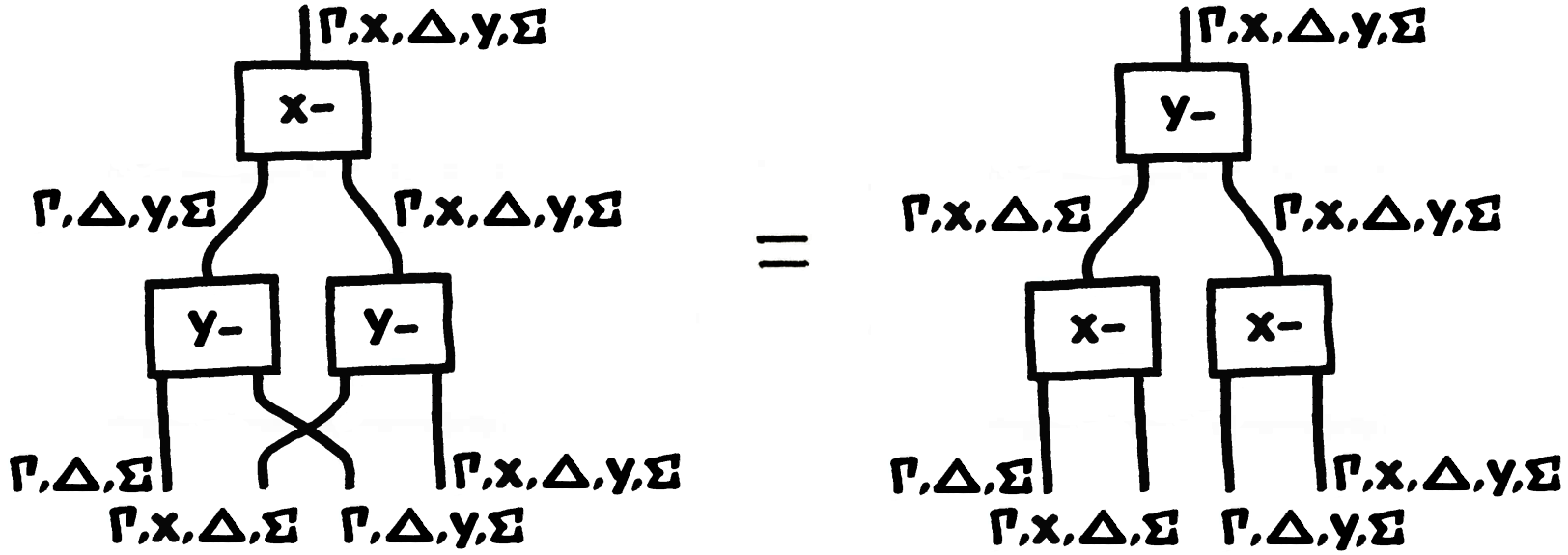}
\end{mathpar}

  
  
\noindent The axioms \textbf{[N4]} through \textbf{[N10]} are of a similar character. The axioms \textbf{[N1]} and \textbf{[N2]} are different, serving to ensure that $N$ is isomorphic to $I + N$. Notice that the none of the equations of $\A$ change the operational effect of the corresponding abacus program.

We proceed to show that $\A$ is an Elgot category. First, we require another monoidal structure $(\A,\otimes,I)$. We begin with the action of $\otimes$ on objects. Of course, $I$ will be the empty monomial. We define $UQ = U \otimes Q$ for $U$ monomial and $P$ polynomial by induction on $P$, as in:
\begin{mathpar}
  U0 = 0
  
  U(V+Q) = UV + UQ
\end{mathpar}
Next, we define $PQ = P \otimes Q$ for polynomials $P,Q$:
\begin{mathpar}  
  0P = P

  (U+P)Q = UQ + PQ
\end{mathpar}
Put another way, if $P = \Sigma_i U_i$ and $Q = \Sigma_j V_j$ then $P \otimes Q = \Sigma_i (\Sigma_j U_iV_j)$.

The definition of the multiplicative tensor product on morphisms is rather invovled. The first step is to inductively define two endofunctors $(U \ltimes -)$ and $(- \rtimes U)$ on $\A$ for any monomial $U$ as follows:
  \begin{multicols}{2}
  \begin{enumerate}
    \item[] $U \ltimes \s_{V,W} = \s_{UV,W}$
    \item[] $U \ltimes \z_{V,W} = \z_{UV,W}$
    \item[] $U \ltimes \p_{V,W} = \p_{UV,W}$
    \item[] $U \ltimes 1_V = 1_{UV}$
    \item[] $U \ltimes \eta_V = \eta_{UV}$
    \item[] $U \ltimes \mu_V = \mu_{UV}$
    \item[] $U \ltimes 1_0 = 1_0$
    \item[] $U \ltimes \sigma^+_{V,W} = \sigma^+_{UV,UW}$
    \item[] $U \ltimes fg = (U \ltimes f)(U \ltimes g)$
    \item[] $U \ltimes f + g = (U \ltimes f) + (U \ltimes g)$
    \item[] $U \ltimes \mathsf{Tr}^W_{P,Q} = \mathsf{Tr}^{UW}_{UP,UQ}(U \ltimes f)$
    \item[] $\s_{V,W} \rtimes U = \s_{V,WU}$
    \item[] $\z_{V,W} \rtimes U = \z _{V,WU}$
    \item[] $\p_{V,W} \rtimes U = \p_{V,WU}$
    \item[] $1_V \rtimes U = 1_{VU}$
    \item[] $\eta_V \rtimes U = \eta_{VU}$
    \item[] $\mu_V \rtimes U = \mu_{VU}$
    \item[] $1_0 \rtimes U = 1_0$
    \item[] $\sigma^+_{V,W} \rtimes U = \sigma^+_{VU,WU}$
    \item[] $fg \rtimes U = (f \rtimes U)(g \rtimes U)$
    \item[] $f + g \rtimes U = (f \rtimes U) + (g \rtimes U)$
    \item[] $\mathsf{Tr}^W_{P,Q} \rtimes U = \mathsf{Tr}^{WU}_{PU,QU}(f \rtimes U)$
  \end{enumerate}
  \end{multicols}
  \noindent We call $(U \ltimes -)$ and $(- \rtimes U)$ \emph{monomial left whiskering} and \emph{monomial right whiskering}, respectively.

  We also require the left distributors. While the right distributors in $\A$ will be identities, the left distributors $\delta^l_{P,Q,R}$ are defined by induction on $P$ as in:
  \begin{mathpar}
    \delta^l_{0,Q,R} = 1_0
    
    \delta^l_{U+P,Q,R} = (1_{UQ+UR} + \delta^l_{P,Q,R})(1_{UQ} + \sigma^+_{UR,PQ} + 1_{PR})
  \end{mathpar}
  \noindent Notice that for monomials $U$ we have $\delta^l_{U,Q,R} = \delta^l_{U+0,Q,R} = 1_{UP + UQ}$. 

  Now, for any polynomial $P$ in $\A$ we define an endofunctor $(P \ltimes -)$ of $\A$ called \emph{polynomial left whiskering}, defined on $f : R \to S$ as in:
  \begin{mathpar}
    0 \ltimes f = 1_0

    U + P \ltimes f = (U \ltimes f) + (P \ltimes f)
  \end{mathpar}
  \noindent Similarly, we define an endofunctor $(- \rtimes P)$ of $\A$ called \emph{polynomial right whiskering}, defined on $f : R \to S$ as in:
  \begin{mathpar}
    f \rtimes 0 = 1_0
    
    f \rtimes U + P = \delta^l_{R,U,P}((f \rtimes U) + (f \rtimes P))(\delta^l_{S,U,P})^{-1}
  \end{mathpar}
Finally, the action of the multiplicative tensor product on morphsims $f : R \to S$ and $g : R' \to S'$ is given by $f \otimes g = (f \rtimes R')(S \ltimes g) = (R \ltimes g)(f \rtimes S')$ and we have:
\begin{theorem}\label{thm:abacus-elgot}
  $(\A,(+,0,\sigma^+),(\otimes,I),(\mu,\eta),\mathsf{Tr},(1,1),(\delta^l,1),(N,(\z_{I,I}+\s_{I,I})\mu_N))$ is a right-strict Elgot category.
\end{theorem}
\begin{proof}
  See Appendix~\ref{sec:proof-appendix}.
\end{proof}

The construction of the category $\A$ and the proof of Theorem~\ref{thm:abacus-elgot} are adapted from the construction of the category of tape diagrams over a monoidal signature and the proof that it forms a rig category with finite biproducts found in Section 5 of~\cite{Bonchi2023}. Therein, the authors begin with a monoidal signature $\Sigma$, construct the symmetric monoidal category $C_\Sigma$ that it presents, and then freely add finite biproducts to the underlying category. The finite biproducts become the additve structure of the resulting rig category, while the multiplicative strucutre is essentially inherited from the monoidal structure of $C_\Sigma$. While this approach can also be used to construct a distributive monoidal category by freely adding finite coproducts instead of finite biproducts, it is not entirely suitable for the situation considered in this paper. In particular, we cannot include the generating morphisms $\p_{U,V} : UNV \to UV + UNV$ as part of a monoidal signature since their codomain involves the additive structure. We instead perform the construction in a single step, relying on the axioms $\textbf{[N3-N11]}$ to ensure that the result is a distributive monoidal category. Curiously, for both the original construction of~\cite{Bonchi2023} and our construction here to work in the case where the additive structure is traced, the trace operator must be assumed to be uniform. This is also observed in more recent work on tape diagrams~\cite{Bonchi2024} involving the trace.

The promised initiality result follows almost immediately:
\begin{theorem}\label{thm:abacus-initial}
$(\A,(+,0,\sigma^+),(\otimes,I),(\mu,\eta),\mathsf{Tr},(1,1),(\delta^l,1),(N,(\z_{I,I}+\s_{I,I})\mu_N))$ is initial in the category of right-strict Elgot categories and functors between them that strictly preserve the Elgot category structure. 
\end{theorem}
\begin{proof}
  Say we have another right-strict Elgot category $(\C,(+,0,\sigma^+),(\otimes,I),(\mu,\eta),\mathsf{Tr},(1,1),(\delta^l,1),(N,\iota))$. Any functor $F : \A \to \C$ that strictly preserves the Elgot category structure must have $FN = N$. Since $F$ preserves the multiplicative monoidal structure this means that $F(N^n) = F(N)^n$, and so the value of the object mapping of $F$ on monomials is forced. The object mapping of $F$ on the polynomials is forced because $F$ is assumed to strictly preserve the cocartesian monoidal structure.

  On morphisms, the requirements that $F$ be a functor, strictly preserves the cocartesian monoidal structure, and strictly preserves the trace operator means that the only choice we might have is where to map the $\z$, $\s$, and $\p$ arrows. Since $F$ strictly preserves the distinguished object $N$ and isomorphism $\iota : I+N \to N$ we must have t $F(\s_{I,I}) = \ip_1\iota$. Together with the assumption that $F$ strictly preserves the multiplicative monoidal structure this means that we must have $F(\s_{U,V}) = F(1_U \otimes \s_{I,I} \otimes 1_V) = 1_U \otimes \ip_1\iota = 1_V$. Similarly, we must have $F(\z_{U,V}) = 1_U \otimes \z_{I,I} \otimes 1_V$. Since $F$ is a functor and $\p_{U,V}$ is inverse to $(\z_{U,V} + \s_{U,V})\mu_{UNV}$ in $\A$, we know that $F(\p_{U,V})$ must be the inverse of $F((\s_{U,V} + \z_{U,V})\mu_{UNV})$, and so the definition of the arrow mapping of $F$ is forced. The claim follows.
\end{proof}

Finally, our category of abacus programs strongly represents the partial recursive functions.
\begin{theorem}\label{thm:abacus-strong-representability}
  $\A$ strongly represents the partial recursive functions.
\end{theorem}
\begin{proof}
  We show that the conditions of Theorem~\ref{thm:plotkin-strong-representable} are satisfied. By assumption we have that $\iota : I + N \to N$ is an isomorphism. Obviously $\underline{0} \neq \underline{1}$ in $\A$. Notice that the partial function $\overline{f} : \mathbb{N}^n \to \mathbb{N}$ strongly represented by some morphism $f : N^n \to N$ of $\A$ coincides with the partial function implemented by the associated abacus program. It follows that any such $\overline{f}$ is partial recursive. That the remaining conditions are satisfied follows from Lemma~\ref{lem:pre-elgot-representable}.
\end{proof}

In particular, since any partial function that is strongly representable in $\A$ is partial recursive, we have:
\begin{corollary}\label{cor:final}
$\A$ strongly represents all and only the partial recursive functions.
\end{corollary}

\section{Conclusions Future Work}\label{sec:conclusions}
We have introduced (pre-)Elgot categories and shown that any such category represents the partial recursive functions. Moreover, we have constructed an initial Elgot category of abacus programs, and have shown that it \emph{strongly} represents the partial recursive functions, and so, in a sense, generates them.

We end by mentioning a few directions for future work. First, it seems very sensible to impose further assumptions on Elgot categories to the effect that the multiplicative monoidal structure forms a cartesian restriction category (see e.g.,~\cite{Cockett2007,DiLiberti2021,Nester2024}). This more structured setting would more faithfully represent the partial recursive functions, and one imagines that stronger results would be obtainable. In particular, it seems likely that the initial such category is a Turing category~\cite{Cockett08}.

Second, there are natural questions concerning the relationship between Elgot categories and Elgot monads~\cite{Adamek2010}. The identity functor on any Elgot category is an Elgot monad, and we might for example ask when the Kleisli category of an Elgot monad is an Elgot category. Moreover, the reformulation of Elgot monads as \emph{while monads} given by Goncharov~\cite{Goncharov2022} suggests that the ideas presented here might lead to a categorical semantics for at least some of the ``While'' languages one encouters in computer science.

\bibliographystyle{./entics}
\bibliography{citations}

\appendix

\section{Distributive Monoidal Categories}\label{sec:distributive-monoidal-categories}
\begin{definition}
  A \emph{distributive monoidal category} is a tuple $(\C,(+,0,\sigma^+),(\otimes,I),(\mu,\eta),(\lambda^\bullet,\rho^\bullet),(\delta^l,\delta^r))$ such that:
  \begin{itemize}
  \item $((\C,+,0,\sigma^+),(\mu,\eta))$ is a cocartesian monoidal category.
  \item $(\C,\otimes,I)$ is a monoidal category.
  \item $\lambda^\bullet$ and $\rho^\bullet$ are natural isomorphisms with components $\lambda^\bullet_A : 0 \otimes A \to 0$ and $\rho^\bullet_A : A \otimes 0 \to 0$ called the left and right \emph{annihilator}, respectively.
  \item $\delta^l$ and $\delta^r$ are natural isomorphisms with components
    \begin{mathpar}
      \delta^l_{A,B,C} : A \otimes (B + C) \to (A \otimes B) + (A \otimes C)

      \delta^r_{A,B,C} : (A + B) \otimes C \to (A \otimes C) + (B \otimes C)
    \end{mathpar}
    called the left and right \emph{distributor}, respectively.
  \item The coherence axioms \textbf{[B1]}-\textbf{[B22]} are satisfied. Before enumerating these axioms, we recall a common convention: When referring to the objects of a distributive monoidal category we will denote the multiplicative tensor product $\otimes$ by juxtaposition. For example, $A \otimes B$ becomes $AB$, $A \otimes B \otimes C$ becomes $ABC$, and $A \otimes (B + (C \otimes D))$ becomes $A(B + CD)$. The coherence axioms ask that for all objects $A,B,C,D$ of $\C$, we have:
    \begin{mathparpagebreakable}
          \begin{tikzcd}
      A(B+ C) \ar[rd,phantom,"{\textbf{[B1]}}"] \ar[d,"1_A \otimes \sigma^+_{B,C}"'] \ar[r,"\delta^l_{A,B,C}"] & AB + AC \ar[d,"\sigma^+_{AB,AC}"] \\
      A(C+ B) \ar[r,"\delta^l_{A,C,B}"'] & AC + AB
    \end{tikzcd}

    \begin{tikzcd}[column sep=large]
      (A + B + C)D \ar[rd,phantom,"\textbf{[B3]}"] \ar[r,"\delta^r_{A+ B,C,D}"] \ar[d,"\delta^r_{A,B + C,D}"'] & (A + B)D + CD  \ar[d,"\delta^r_{A,B,D} + 1_{CD}"] \\
      AD + (B + C)D \ar[r,"1_{AD} + \delta^r_{B,C,D}"'] & AD + BD + CD
    \end{tikzcd}

    \begin{tikzcd}
      (A + B)C \ar[rd,phantom,"{\textbf{[B2]}}"] \ar[d,"\sigma^+_{A,B} \otimes 1_C"'] \ar[r,"\delta^r_{A,B,C}"] & AC + BC \ar[d,"\sigma^+_{AC,BC}"] \\
      (B + A)C \ar[r,"\delta^r_{B,A,C}"'] & BC + AC
    \end{tikzcd}

    \begin{tikzcd}[column sep=large]
      A(B + C + D) \ar[d,"\delta^l_{A,B,C + D}"'] \ar[r,"\delta^l_{A,B+ C,D}"] \ar[rd,phantom,"\textbf{[B4]}"] & A(B + C) + AD \ar[d,"\delta^l_{A,B,C} + 1_{AD}"] \\
      AB + A(C + D) \ar[r,"1_{AB} + \delta^l_{A,C,D}"'] & AB + AC + AD
    \end{tikzcd}
    
    \begin{tikzcd}
      AB(C + D) \ar[r,"\delta^l_{AB,C,D}"] \ar[d,"1_A \otimes \delta^l_{B,C,D}"'] \ar[rd,phantom,"\textbf{[B5]}"] & ABC + ABD \ar[d,equal]\\
      A(BC + BD) \ar[r,"\delta^l_{A,BC,BD}"'] & ABC + ABD
    \end{tikzcd}

    \begin{tikzcd}
      (A + B)CD \ar[d,"\delta^r_{A,B,C} \otimes 1_D"'] \ar[r,"\delta^r_{A,B,CD}"] \ar[rd,phantom,"\textbf{[B6]}"] & ACD + BCD \ar[d,equal] \\
      (AC + BC)D \ar[r,"\delta^r_{AC,BC,D}"'] & ACD + BCD
    \end{tikzcd}

    \begin{tikzcd}
      A(B + C)D \ar[r,"\delta^l_{A,B,C} \otimes 1_D"] \ar[d,"1_A \otimes \delta^r_{B,C,D}"'] \ar[rd,phantom,"\textbf{[B7]}"] & (AB + AC)D \ar[d,"\delta^r_{AB,AC,D}"] \\
      A(BD + CD) \ar[r,"\delta^l_{A,BD,CD}"'] & ABD + ACD
    \end{tikzcd}

    \begin{tikzcd}
      & A + B + C + D \ar[rd,"\delta^r_{A,B,C+ D}"] \ar[dl,"\delta^l_{A + B,C,D}"'] \ar[dd,phantom,"\textbf{[B8]}"]  \\
      (A + B)C + (A + B)D \ar[d,"\delta^r_{A,B,C} + \delta^r_{A,B,D}"'] && A(C + D) + B(C + D) \ar[d,"\delta^l_{A,C,D} + \delta^l_{B,C,D}"] \\
      AC + BC + AD + BD \ar[rr,"1_{AC} + \sigma^+_{BC,AD} + 1_{BD}"'] & \text{} & AC + AD + BC + BD
    \end{tikzcd}

    \begin{tikzcd}
      00 \ar[rr,phantom,"\textbf{[B9]}"] \ar[rr,bend left=60,"\lambda^\bullet_0"] \ar[rr,bend right=60,"\rho^\bullet_0"'] && 0
    \end{tikzcd}
    
    \begin{tikzcd}
      0(A + B) \ar[rd,phantom,"\textbf{[B10]}"] \ar[d,"\lambda^\bullet_{A+ B}"']  \ar[r,"\delta^l_{0,A,B}"] & 0A + 0B \ar[d,"\lambda^\bullet_A + \lambda^\bullet_B"] \\
      0 & 0 + 0 \ar[l,equal] 
    \end{tikzcd}

    \begin{tikzcd}
      (A + B)0 \ar[rd,phantom,"\textbf{[B11]}"] \ar[r,"\delta^r_{A,B,0}"] \ar[d,"\rho^\bullet_{A+ B}"'] & A0 + B0 \ar[d,"\rho^\bullet_A + \rho^\bullet_B"] \\
      0 & 0 + 0 \ar[l,equal] 
    \end{tikzcd}

    \begin{tikzcd}
      0I \ar[rr,phantom,"\textbf{[B12]}"] \ar[rr,bend left=60,"\lambda^\bullet_I"] \ar[rr,bend right=60,equal] && 0
    \end{tikzcd}

    \begin{tikzcd}
      I0 \ar[rr,bend left=60,"\rho^\bullet_I"] \ar[rr,bend right=60,equal] \ar[rr,phantom,"\textbf{[B13]}"] && 0 
    \end{tikzcd}
    \\
    \begin{tikzcd}
      AB0 \ar[r,"\rho^\bullet_{AB}"] \ar[d,"1_A \otimes \rho^\bullet_B"'] \ar[rd,phantom,"\textbf{[B14]}"] & 0 \ar[d,equal]\\
      A0 \ar[r,"\rho^\bullet_A"'] & 0
    \end{tikzcd}

    \begin{tikzcd}
      A(0 + B) \ar[rd,phantom,"\textbf{[B17]}"] \ar[d,equal] \ar[r,"\delta^l_{A,0,B}"] & A0 + AB \ar[d,"\rho^\bullet_A + 1_{AB}"] \\
      AB & 0 + AB \ar[l,equal] 
    \end{tikzcd}

    \begin{tikzcd}
      (0 + B)A \ar[d,equal] \ar[rd,phantom,"\textbf{[B18]}"] \ar[r,"\delta^r_{0,B,A}"] & 0A + BA \ar[d,"\lambda^\bullet_A + 1_{BA}"] \\
      BA & 0 + BA \ar[l,equal]
    \end{tikzcd}

    \begin{tikzcd}
      A0B \ar[r,"1_A \otimes \lambda^\bullet_B"] \ar[d,"\rho^\bullet_A \otimes 1_B"'] \ar[rd,phantom,"\textbf{[B15]}"] & A0 \ar[d,"\rho^\bullet_A"] \\
      0B \ar[r,"\lambda^\bullet_B"'] & 0
    \end{tikzcd}
    
    \begin{tikzcd}
      A(B + 0) \ar[rd,phantom,"\textbf{[B19]}"] \ar[r,"\delta^l_{A,B,0}"] \ar[d,equal] & AB + A0 \ar[d,"1_{AB} + \rho^\bullet_A"] \\
      AB & AB + 0 \ar[l,equal] 
    \end{tikzcd}

    \begin{tikzcd}
      (B + 0)A \ar[rd,phantom,"\textbf{[B20]}"] \ar[r,"\delta^r_{B,0,A}"] \ar[d,equal] & BA + 0A \ar[d,"1_{BA} + \lambda^\bullet_{A}"] \\
      BA & BA + 0 \ar[l,equal] 
    \end{tikzcd}

    \begin{tikzcd}
      0AB \ar[r,"\lambda^\bullet_{AB}"] \ar[d,"\lambda^\bullet_A \otimes 1_B"'] \ar[rd,phantom,"\textbf{[B16]}"] & 0 \ar[d,equal] \\
      0B \ar[r,"\lambda^\bullet_B"'] & 0
    \end{tikzcd}
    
    \begin{tikzcd}
      I(A + B) \ar[d,equal] \ar[rd,phantom,"\textbf{[B21]}"] \ar[r,"\delta^l_{I,A,B}"] & IA + IB \ar[d,equal] \\
      A + B \ar[r,equal] & A + B
    \end{tikzcd}
    
    \begin{tikzcd}
      (A + B)I \ar[rd,phantom,"\textbf{[B22]}"] \ar[d,equal] \ar[r,"\delta^r_{A,B,I}"] & AI + BI \ar[d,equal]  \\
      A + B \ar[r,equal] & A + B
    \end{tikzcd}
    \end{mathparpagebreakable}
  \end{itemize}
\end{definition}

\begin{definition}
A distributive monoidal category $(\C,(+,0,\sigma^+),(\otimes,I),(\mu,\eta),(\lambda^\bullet,\rho^\bullet),(\delta^l,\delta^r))$ is said to be \emph{right-strict} in case $\lambda^\bullet$,$\rho^\bullet$, and $\delta^r$ are identity natural transformations. Explicitly, a \emph{right-strict distributive monoidal category} is a tuple
  \[
  (\C,(+,0,\sigma^+),(\otimes,I),(\mu,\eta),\delta^l)
  \]
  such that:
  \begin{itemize}
  \item $(\C,+,0,\sigma^+)$ is a symmetric strict monoidal category.
  \item $(\C,\otimes,I)$ is a strict monoidal category.
  \item $\delta^l$ is a natural isomorphism with components $\delta^l_{A,B,C} : A(B + C) \to AB + AC$. That is, for all $f : A \to A', g : B \to B', h : C \to C'$ of $\C$ we must have:
    \begin{mathpar}
      \begin{tikzcd}
        A(B + C) \ar[r,"\delta^l_{A,B,C}"] \ar[d,"f \otimes (g + h)"'] & AB + AC \ar[d,"(f \otimes g) + (f \otimes h)"] \\ 
        A'(B' + C') \ar[r,"\delta^l_{A',B',C'}"'] & A'B' + A'C'
      \end{tikzcd}
    \end{mathpar}
  \item For all objects $A,B,C$ of $\C$, we have:
    \begin{mathpar}
      A0 = 0
      
      0A = 0
      
      (A + B)C = AC + BC
    \end{mathpar}
    Similarly, for all morphisms $f,g,h$ of $\C$ we have:
    \begin{mathpar}
      f \otimes 1_0 = 1_0

      1_0 \otimes f = 1_0
      
      (f + g) \otimes h = (f \otimes h) + (g \otimes h)
    \end{mathpar}
  \item The following equations are satisfied for all objects $A,B,C,D$ of $\C$:
    \begin{mathpar}
      \delta^l_{0,A,B} = 1_0

      \delta^l_{A,0,B} = 1_{AB}

      \delta^l_{A,B,0} = 1_{AB}

      \delta^l_{I,A,B} = 1_{A + B}

      \delta^l_{A,B,C} \sigma^+_{AB,AC}  = (1_A \otimes \sigma^+_{B,C})\delta^l_{A,C,B}

      \sigma^+_{AC,BC} = \sigma^+_{A,B} \otimes 1_C

      \delta^l_{A,B+ C,D}(\delta^l_{A,B,C} + 1_{AD}) = \delta^l_{A,B,C + D}(1_{AB} + \delta^l_{A,C,D})

      \delta^l_{AB,C,D} = (1_A \otimes \delta^l_{B,C,D})\delta^l_{A,BC,BD}

      \delta^l_{A,BD,CD} = \delta^l_{A,B,C} \otimes 1_D

      \delta^l_{A + B,C,D} = (\delta^l_{A,C,D} + \delta^l_{B,C,D})(1_{AC} + \sigma^+_{AD,BC} + 1_{BD})
    \end{mathpar}
    
  \end{itemize}
\end{definition}

\section{Proof of Theorem~\ref{thm:abacus-elgot}}\label{sec:proof-appendix}
This section contains the proof of Theorem~\ref{thm:abacus-elgot}, restated here as Theorem~\ref{thm:appendix-abacus-elgot}. First, notice that by construction, we have:
\begin{lemma}
  \begin{itemize}
  \item $((\A,+,0,\sigma^+),(\mu,\eta))$ is a cocartesian monoidal category.
  \item $((\A,+,0,\sigma^+),\mathsf{Tr})$ is a uniform traced monoidal category.
  \end{itemize}
\end{lemma}
To show that $\A$ is an Elgot category, we must show that our proposed multiplicative monoidal structure satisfies the axioms of a monoidal category, and that it distributes over the existing additive structure. We begin by showing that our $\otimes$ operation forms a monoid on the objects of $\A$:
\begin{lemma}
  For all polynomials $P,Q,R$ and monomials $U$:
  \begin{enumerate}
  \item $P0 = 0$
  \item $U(Q+R) = UQ + UR$
  \item $(P+Q)R = PR + QR$
  \item $P(QR) = (PQ)R$
  \item $IP = P$
  \item $PI = P$
  \end{enumerate}
\end{lemma}
\begin{proof}
Straightforward structural induction.
\end{proof}

Next, we establish some elementary properties of monomial whiskering:
\begin{lemma}\label{lem:monomial-whiskering-properties}
  For all monomials $U$ and polynomials $Q,R,S$, we have:
  \begin{multicols}{2}
  \begin{enumerate}
  \item $U \ltimes 1_Q = 1_{UQ}$
  \item $U \ltimes \sigma^+_{Q,R} = \sigma^+_{UQ,UR}$
  \item $U \ltimes \eta_Q = \eta_{UQ}$
  \item $U \ltimes \mu_Q = \mu_{UQ}$
  \item $U \ltimes \mathsf{Tr}^Q_{R,S}(f) = \mathsf{Tr}^{UR}_{UP,UQ}(U \ltimes f)$
  \item $1_Q \rtimes U = 1_{QU}$
  \item $\sigma^+_{P,Q} \rtimes U = \sigma^+_{PU,QU}$
  \item $\eta_Q \rtimes U = \eta_{QU}$
  \item $\mu_Q \rtimes U = \mu_{QU}$
  \item $\mathsf{Tr}^R_{P,Q}(f) \rtimes U = \mathsf{Tr}^{RU}_{PU,QU}(f \rtimes U)$
  \end{enumerate}
  \end{multicols}
\end{lemma}
\begin{proof}
  \begin{enumerate}
  \item By induction on $Q$. The base case is $U \ltimes 1_0 = 1_0 = 1_{U0}$. For the inductive case suppose $U \ltimes 1_Q = 1_{UQ}$. Then we have:
    \begin{align*}
      & U \ltimes 1_{V + Q}
      = U \ltimes (1_V + 1_Q)
      = (U\ltimes 1_V) + (U\ltimes 1_Q)
      = 1_{UV} + 1_{UQ}
      = 1_{U(V+Q)}
    \end{align*}
    and the claim follows.
  \item First, we show that $U \ltimes \sigma^+_{V,R} = \sigma^+_{UV,UR}$ by induction on $R$. The base case is $U \ltimes \sigma^+_{V,0}  = U \ltimes 1_V = 1_{UV} = \sigma^+_{UV,U0}$. For the inductive case, suppose $U \ltimes \sigma^+_{V,R} = \sigma^+_{UV,UR}$. Then we have:
    \begin{align*}
      & U \ltimes \sigma^+_{V,W+R}
      = U \ltimes (\sigma^+_{V,W} + 1_R)(1_W + \sigma^+_{V,R})
      \\&= ((U \ltimes \sigma^+_{V,W}) + 1_{UR})(1_{UW} + (U \ltimes \sigma^+_{V,R}))
      \\&= (\sigma^+_{UV,UW} + 1_{UR})(1_{UW} + \sigma^+_{UV,UR})
      \\&= \sigma^+_{UV,UW+UR} = \sigma^+_{UV,U(W+R)}
    \end{align*}
    and it follows that $U \ltimes \sigma^+_{V,R} = \sigma^+_{UV,UR}$. We proceed to prove the full claim by induction on $Q$. The base case is $U \ltimes \sigma^+_{0,R} = U \ltimes 1_R = 1_{UR} = \sigma^+_{U0,UR}$. For the inductive case, suppose $U \ltimes \sigma^+_{Q,R} = \sigma^+_{UQ,UR}$. Then we have:
    \begin{align*}
      & U \ltimes \sigma^+_{V+Q,R}
      = U \ltimes (1_V + \sigma^+_{Q,R})(\sigma^+_{V,R} + 1_Q)
      \\&= (1_{UV} + (U \ltimes \sigma^+_{Q,R}))((U \ltimes \sigma^+_{V,R}) + 1_{UQ})
      \\&= (1_{UV} + \sigma^+_{UQ,UR})(\sigma^+_{UV,UR} + 1_{UQ})
      \\&= \sigma^+_{UV+UQ,UR} = \sigma^+_{U(V+Q),UR}
    \end{align*}
    and the claim follows.
  \item By induction on $Q$. The base case is $U \ltimes \eta_0 = U \ltimes 1_0 = 1_0 = \eta_0 = \eta_{U0}$. For the inductive case, suppose $U \ltimes \eta_Q = \eta_{UQ}$. Then we have:
    \begin{align*}
      & U \ltimes \eta_{V + Q}
      = U \ltimes (\eta_V + \eta_Q)
      = (U \ltimes \eta_V) + (U \ltimes \eta_Q)
      = \eta_{UV} + \eta_{UQ}
      = \eta_{U(V+Q)}
    \end{align*}
    and the claim follows.
  \item By induction on $Q$. The base case is $U \ltimes \mu_0 = U \ltimes 1_0 = 1_0 = \mu_0 = \mu_{U0}$. For the inductive case, suppose $U \ltimes \mu_Q = \mu_{UQ}$. Then we have:
    \begin{align*}
      & U \ltimes \mu_{V+Q}
      = U \ltimes (1_V + \sigma^+_{Q,V} + 1_Q)(\mu_V + \mu_Q)
      \\&= (1_{UV} + \sigma^+_{UQ,UV} + 1_{UQ})((U \ltimes \mu_V) + (U \ltimes \mu_Q))
      \\&= (1_{UV} + \sigma^+_{UQ,UV} + 1_{UQ})(\mu_{UV} + \mu_{UQ})
      \\&= \mu_{UV + UQ}
      = \mu_{U(V+Q)}
    \end{align*}
    and the claim follows.
      \item The base case is $U \ltimes \mathsf{Tr}^0_{R,S}(f) = U \ltimes f = \mathsf{Tr}^{U0}_{UR,US}(U \ltimes f)$. For the inductive case, suppose $U \ltimes \mathsf{Tr}^Q_{R',S'}(f) = \mathsf{Tr}^{UQ}_{UR',US'}(U \ltimes f)$ for all polynomials $R',S'$. Then we have:
    \begin{align*}
      & U \ltimes \mathsf{Tr}^{V + Q}_{R,S}(f)
      = U \ltimes \mathsf{Tr}^V_{R,S}(\mathsf{Tr}^Q_{R+V,S+V}(f))
      \\&= \mathsf{Tr}^{UV}_{UR,US}(U \ltimes \mathsf{Tr}^Q_{R+V,S+V}(f))
      = \mathsf{Tr}^{UV}_{UR,US}(\mathsf{Tr}^{UQ}_{U(R+V),S(R+V)}(U \ltimes f))
      \\&= \mathsf{Tr}^{UV+UQ}_{UR,US}(U \ltimes f)
      = \mathsf{Tr}^{U(V+Q)}_{UR,US}(U \ltimes f)
    \end{align*}
    and the claim follows.
  \item By induction on $Q$. The base case is $1_0 \rtimes U = 1_0 = 1_{0U}$. For the inductive case, suppose $1_P \rtimes U = 1_{PU}$. Then we have:
    \begin{align*}
    & 1_V \rtimes (U+P)
      = \delta^l_{V,U,P}((1_V \rtimes U) + (1_V \rtimes P))(\delta^l_{V,U,P})^{-1}
      \\&= 1_{VU+VP}(1_{VU} + 1_{VP})1_{VU+VP}
      = 1_{V(U+P)}
    \end{align*}
    and the claim follows.
  \item First, we show $\sigma^+_{V,Q} \rtimes U = \sigma^+_{VU,QU}$ for all monomials $V$, polynomials $Q$. We proceed by induction on $Q$. The base case is $\sigma^+_{V,0} \rtimes U = 1_V \rtimes U = 1_{VU} = \sigma^+_{VU,0U}$. For the inductive case, suppose $\sigma^+_{V,Q} \rtimes U = \sigma^+_{VU,QU}$. Then we have:
    \begin{align*}
      & \sigma^+_{V,W + Q} \rtimes U
      = (\sigma^+_{V,W} + 1_Q)(1_W + \sigma^+_{V,Q}) \rtimes U
      \\&= (\sigma^+_{VU,WU} + 1_{QU})(1_{WU} + \sigma^+_{VU,QU})
      = \sigma^+_{VU,WU + QU}
      = \sigma^+_{VU,(W+Q)U}
    \end{align*}
    and it follows that $\sigma^+_{V,Q} \rtimes U = \sigma^+_{VU,QU}$ for all monomials $V$, polynomials $Q$. We are now ready to prove the claim, which we do by induction on $P$. The base case is $\sigma^+_{0,Q} \rtimes U = 1_Q \rtimes U = 1_{QU} = \sigma^+_{0U,QU}$. For the inductive case, suppse $\sigma^+_{P,Q} \rtimes U = \sigma^+_{PU,QU}$. Then we have:
    \begin{align*}
      & \sigma^+_{V+P,Q} \rtimes U
      = (1_V + \sigma^+_{P,Q})(\sigma^+_{V,Q} + 1_P) \rtimes U
      \\&= (1_{VU} + \sigma^+_{PU,QU})(\sigma^+_{VU,QU} + 1_{PU})
      = \sigma^+_{VU + PU,QU}
      = \sigma^+_{(V+P)U,QU}
    \end{align*}
    and the claim follows.
  \item By induction on $Q$. The base case is $\eta_0 \rtimes U = 1_0 \rtimes U = 1_0 = \eta_0 = \eta_{0U}$. For the inductive case, suppose $\eta_Q \rtimes U = \eta_{QU}$. Then we have:
    \begin{align*}
      & \eta_{V + Q} \rtimes U
      = (\eta_{V} + \eta_{Q})\rtimes U
      = (\eta_{V} \rtimes U) + (\eta_{Q} \rtimes U)
      = \eta_{VU} + \eta_{QU}
      = \eta_{(V+Q)U}
    \end{align*}
    and the claim follows.
  \item By induction on $Q$. The base case is $\mu_0 \rtimes U = 1_0 \rtimes U = 1_0 = \mu_0 = \mu_{0U}$. For the inductive case, suppose $\mu_Q \rtimes U = \mu_{QU}$. Then we have:
    \begin{align*}
      & \mu_{V+Q} \rtimes U
      = (1_V + \sigma^+_{Q,V} + 1_Q)(\mu_V + \mu_Q) \rtimes U
      \\&= (1_{VU} + \sigma^+_{QU,VU} + 1_{QU})(\mu_{VU} + \mu_{QU})
      \\&= \mu_{VU+QU}
      = \mu_{(V+Q)U}
    \end{align*}
    and the claim follows.
      \item By induction on $Q$. The base case is $\mathsf{Tr}^0_{P,Q}(f) \rtimes U = f \rtimes U$. For the inductive case, suppose $\mathsf{Tr}^{R}_{P,Q}(f) \rtimes U = \mathsf{Tr}^{RU}_{PU,QU}(f \rtimes U)$ for all $P,Q$. Then we have:
    \begin{align*}
      & \mathsf{Tr}^{W+R}_{P,Q}(f) \rtimes U
      = \mathsf{Tr}^W_{P,Q}(\mathsf{Tr}^R_{P+W,Q+W}(f)) \rtimes U
      = \mathsf{Tr}^{WU}_{PU,QU}(\mathsf{Tr}^R_{P+W,Q+W}(f) \rtimes U)
      \\&= \mathsf{Tr}^{WU}_{PU,QU}(\mathsf{Tr}^{RU}_{(P+W)U,(Q+W)U}(f \rtimes U))
      = \mathsf{Tr}^{WU}_{PU,QU}(\mathsf{Tr}^{RU}_{PU+WU,QU+WU}(f \rtimes U))
      \\&= \mathsf{Tr}^{WU + RU}_{PU,QU}(f \rtimes U)
      = \mathsf{Tr}^{(W+R)U}_{PU,QU}(f \rtimes U)
    \end{align*}
    and the claim follows.
  \end{enumerate}
\end{proof}

The analogous properties of polynoimial whiskering are:
\begin{lemma}\label{lem:polynomial-whiskering-properties}
  For polynomial whiksering we have:
  \begin{multicols}{2}
  \begin{enumerate}
   \item $P \ltimes 1_Q = 1_{PQ}$

   \item $(P \ltimes \sigma^+_{Q,R})\delta^l_{P,R,Q} = \delta^l_{P,Q,R}\sigma^+_{PQ,PR}$
     
   \item $P \ltimes \eta_Q = \eta_{PQ}$
     
   \item $P \ltimes \mu_Q = \delta^l_{P,Q,Q}\mu_{PQ}$ 

   \item $P \ltimes \mathsf{Tr}^Q_{R,S}(f) \\ = \mathsf{Tr}^{PQ}_{PR,PS}((\delta^l_{P,R,Q})^{-1}(P \ltimes f)\delta^l_{P,S,Q})$
     
   \item $P \ltimes 1_0 = 1_0$

   \item $P \ltimes ts = (P \ltimes t)(P \ltimes s)$
     
   \item $(P \ltimes t + s)\delta^l_{P,Q_1,Q_2} \\ = \delta^l_{P,P_1,P_2}((P \ltimes t) + (P \ltimes s))$
  
   \item $1_Q \rtimes P = 1_{QP}$
     
   \item $\sigma^+_{P,Q} \rtimes R = \sigma^+_{PR,QR}$

   \item $\eta_Q \rtimes P = \eta_{QP}$

   \item $\mu_Q \rtimes P = \mu_{QP}$

   \item $\mathsf{Tr}^R_{P,Q}(f) \rtimes S \\ = \mathsf{Tr}^{RS}_{PS,QS}(f \rtimes S)$
          
   \item $1_0 \rtimes P = 1_0$

   \item $ts \rtimes P = (t \rtimes P)(s \rtimes P)$
     
   \item $t + s \rtimes P \\ = (t \rtimes P) + (s \rtimes P)$
  \end{enumerate}
  \end{multicols}
\end{lemma}
\begin{proof}
  \begin{enumerate}
  \item By induction on P. The base case is $0 \ltimes 1_Q = 1_0 = 1_{0Q}$. For the inductive case, suppose $P \ltimes 1_Q = 1_{PQ}$. Then using Lemma~\ref{lem:monomial-whiskering-properties} we have:
    \begin{align*}
      & (U + P) \ltimes 1_Q
      = (U \ltimes 1_Q) + (P \ltimes 1_Q)
      = 1_{UQ} + 1_{PQ} = 1_{(U + P)Q}
    \end{align*}
    and the claim follows.
  \item By induction on $P$. The base case is $(0 \ltimes \sigma^+_{Q,R})\delta^l_{0,R,Q} = 1_0 = \sigma^+_{0,0} = \delta^l_{0,Q,R}\sigma^+_{0Q,0R}$. For the inductive case, suppose $(P \ltimes \sigma^+_{Q,R})\delta^l_{P,R,Q} = \delta^l_{P,Q,R}\sigma^+_{PQ,PR}$. Then using Lemma~\ref{lem:monomial-whiskering-properties}we have:
    \begin{align*}
      & ((U + P) \ltimes \sigma^+_{Q,R})\delta^l_{U + P,R,Q}
      \\&= ((U \ltimes \sigma^+_{Q,R}) + (P \ltimes \sigma^+_{Q,R}))(1_{UR+UQ} + \delta^l_{P,R,Q})(1_{UR} + \sigma^+_{UQ,PR} + 1_{PQ})
      \\&= (\sigma^+_{UQ,UR} + (P \ltimes \sigma^+_{Q,R})\delta^l_{P,R,Q})(1_{UR} + \sigma^+_{UQ,PR} + 1_{PQ})
      \\&= (\sigma^+_{UQ,UR} + \delta^l_{P,Q,R}\sigma^+_{PQ,PR})(1_{UR} + \sigma^+_{UQ,PR} + 1_{PQ})
      \\&= (1_{UQ+UR} + \delta^l_{P,Q,R})(1_{UQ} + \sigma^+_{UR,PQ} + 1_{PR})\\&\hspace{0.5cm}(1_{UQ} + \sigma^+_{PQ,UR} + 1_{PR})(\sigma^+_{UQ,UR} + \sigma^+_{PQ,PR})(1_{UR} + \sigma^+_{UQ,PR} + 1_{PQ})
      \\&= \delta^l_{U+P,Q,R}\sigma^+_{UQ+PQ,UR+PR}
      = \delta^l_{U+P,Q,R}\sigma^+_{(U+P)Q,(U+P)R}
    \end{align*}
    and the claim follows.
  \item By induction on $P$. The base case is $0 \ltimes \eta_Q = 1_0 = \eta_0 = \eta_{0Q}$. For the inductive case, suppose $P \ltimes \eta_Q = \eta_{PQ}$. Then using Lemma~\ref{lem:monomial-whiskering-properties} we have:
    \begin{align*}
      & (U + P) \ltimes \eta_Q
      = (U \ltimes \eta_Q) + (P \ltimes \eta_Q)
      = \eta_{UQ} + \eta_{PQ}
      = \eta_{(U+P)Q}
    \end{align*}
    and the claim follows.
  \item By induction on $P$. The base case is $0 \ltimes \mu_Q = 1_0 = \mu_0 = \mu_{0Q}$. For the inductive case, suppose $P \ltimes \mu_Q = \delta^l_{P,Q,Q}\mu_{PQ}$. Then using Lemma~\ref{lem:monomial-whiskering-properties} we have:
    \begin{align*}
      & (U + P) \ltimes \mu_Q
      = (U \ltimes \mu_Q) + (P \ltimes \mu_Q)
      \\&= \mu_{UQ} + \delta^l_{P,Q,Q}\mu_{PQ}
      = (1_{UQ+UQ} + \delta^l_{P,Q,Q})(\mu_{UQ} + \mu_{PQ})
      \\&= (1_{UQ+UQ} + \delta^l_{P,Q,Q})(1_{UQ} + \sigma^+_{UQ,PQ} + 1_{UQ})(1_{UQ} + \sigma^+_{PQ,UQ} + 1_{PQ})(\mu_{UQ} + \mu_{PQ})
      \\&= \delta^l_{U+P,Q,Q}\mu_{(U+P)Q}
    \end{align*}
    and the claim follows.
  \item By induction on $P$. The base case is $0 \ltimes \mathsf{Tr}^Q_{R,S}(f) = 1_0 = \mathsf{Tr}^0_{0,0}(1_0) = \mathsf{Tr}^{0Q}_{0R,0S}((\delta^l_{0,R,Q})^{-1}(0 \ltimes f)\delta^l_{0,S,Q})$. For the inductive case, suppose $P \ltimes \mathsf{Tr}^Q_{R,S}(f) = \mathsf{Tr}^{PQ}_{PR,PS}((\delta^l_{P,R,Q})^{-1}(P \ltimes f)\delta^l_{P,S,Q})$. Then using Lemma~\ref{lem:monomial-whiskering-properties} we have:
    \begin{align*}
      & (U + P) \ltimes \mathsf{Tr}^{Q}_{R,S}(f)
      = (U \ltimes \mathsf{Tr}^{Q}_{R,S}(f)) + (P \ltimes \mathsf{Tr}^Q_{R,S}(f))
      \\&= \mathsf{Tr}^{UQ}_{UR,US}(U \ltimes f) + \mathsf{Tr}^{PQ}_{PR,PS}((\delta^l_{P,R,Q})^{-1}(P \ltimes f)\delta^l_{P,S,Q})
      \\&\stackrel{*}{=} \mathsf{Tr}^{UQ+PQ}_{UR+PR,US+PS}( (1_{UR} + \sigma^+_{PR,UQ} + 1_{PQ})(1_{UR+UQ} + (\delta^l_{P,R,Q})^{-1})\\&\hspace{0.5cm}((U \ltimes f) + (P \ltimes f))(1_{US+UQ} + \delta^l_{P,S,Q})(1_{US} + \sigma^+_{UQ,PS} + 1_{PQ}))
      \\&= \mathsf{Tr}^{(U+P)Q}_{(U+P)R,(U+P)S}((\delta^l_{U+P,R,Q})^{-1}((U+P) \ltimes f)\delta^l_{U+P,S,Q})
    \end{align*}
    where the marked equation ($\stackrel{*}{=}$) holds as in:
    \[
    \includegraphics[height=5cm,align=c]{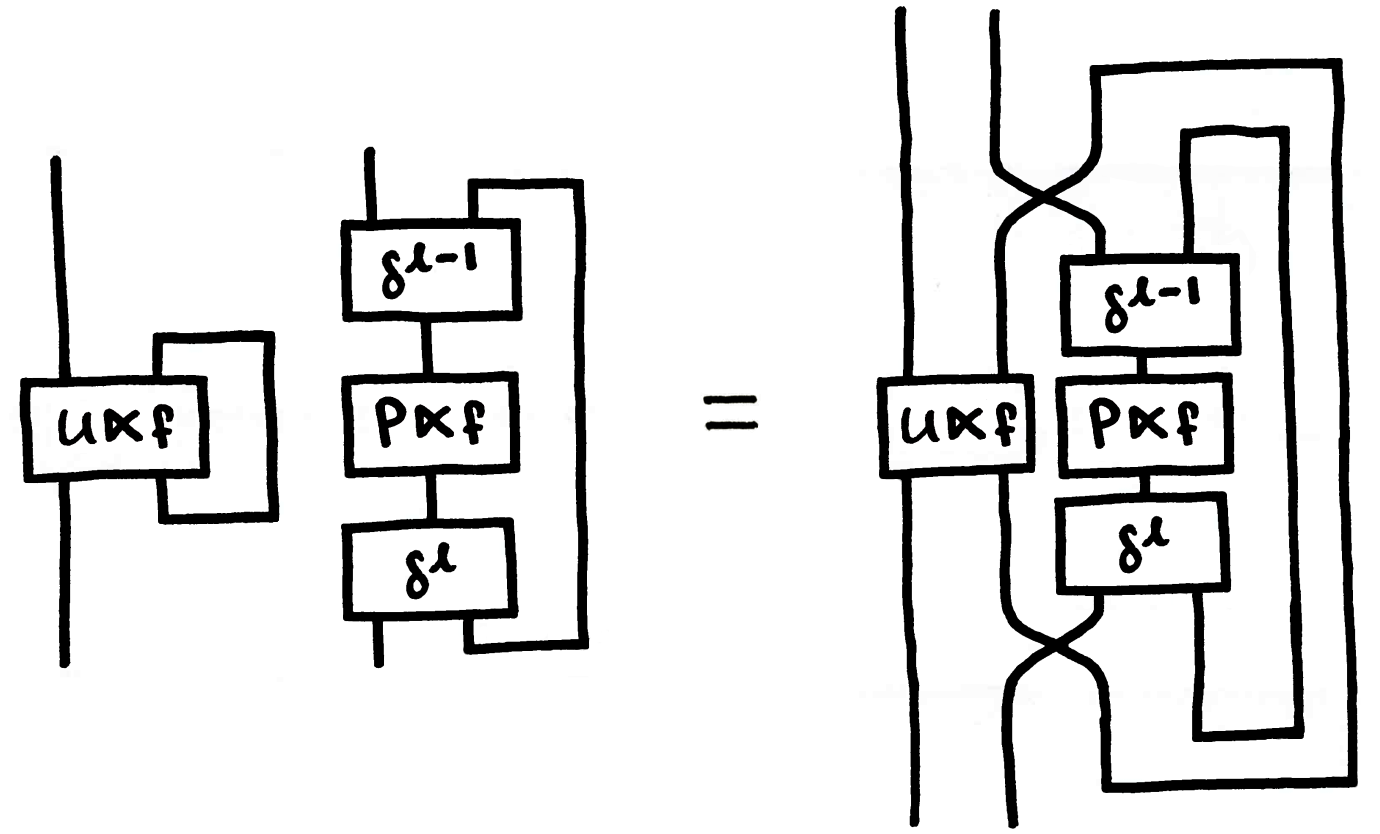}
    \]
    The claim follows.
  \item By induction on $P$. The base case is $0 \ltimes 1_0 = 1_0$. For the inductive case, suppose $P \ltimes 1_0 = 1_0$. Then we have:
    \begin{align*}
      & (U + P) \ltimes 1_0
      = (U \ltimes 1_0) + (P \ltimes 1_0)
      = 1_0 + 1_0 = 1_0
    \end{align*}
    and the claim follows.
  \item By induction on P. The base case is $0 \ltimes ts = 1_0 = (0 \ltimes t)(0 \ltimes s)$. For the inductive case, suppose $P \ltimes ts = (P \ltimes t)(P \ltimes s)$. Then we have:
    \begin{align*}
      & (U + P) \ltimes ts
      = (U \ltimes ts) + (P \ltimes ts)
      = (U \ltimes t)(U \ltimes s) + (P \ltimes t)(P \ltimes s)
      \\&= ((U \ltimes t) + (P \ltimes T))((U \ltimes s) + (P \ltimes S))
      = ((U + P) \ltimes t)((U + P) \ltimes s)
    \end{align*}
    and the claim follows.
  \item By induction on $P$. The base case is $(0 \ltimes (t + s))\delta^l_{0,Q_1,Q_2} = 1_0 = 1_0 + 1_0 = \delta^l_{0,P_1,P_2}((0 \ltimes t) + (0 \ltimes s))$. For the inductive case, suppose $(P \ltimes (t + s))\delta^l_{P,Q_1,Q_2} = \delta^l_{P,P_1,P_2}((P \ltimes t) + (P \ltimes s))$. Then we have:
    \begin{align*}
      & ((U+P) \ltimes (t + s))\delta^l_{U + P,Q_1,Q_2}
      \\&= ((U \ltimes (t + s)) + (P \ltimes (t + s))\delta^l_{P,Q_1,Q_2})(1_{UQ_1} + \sigma^+_{UQ_2,PQ_1} + 1_{PQ_2})
      \\&= ((U \ltimes t) + (U \ltimes s) + \delta^l_{P,P_1,P_2}((P \ltimes t) + (P \ltimes s)))(1_{UQ_1} + \sigma^+_{UQ_2,PQ_1} + 1_{PQ_2})
      \\&= (1_{UP_1 + UP_2} + \delta^l_{P,P_1,P_2})(1_{UP_1} + \sigma^+_{UP_2,PP_1} + 1_{PP_2})((U \ltimes t) + (P \ltimes t) + (U \ltimes s) + (P \ltimes s))
      \\&= \delta^l_{U+P,P_1,P_2}(((U+P) \ltimes t) + ((U+P) \ltimes s))
    \end{align*}
    and the claim follows.
  \item By induction on $P$. The base case is $1_Q \rtimes 0 = 1_0 = 1_{Q0}$. For the inductive case, suppose $1_Q \rtimes P = 1_{QP}$. Then using Lemma~\ref{lem:monomial-whiskering-properties} we have:
    \begin{align*}
      & 1_Q \rtimes (U + P)
      = \delta^l_{Q,U,P}( (1_Q \rtimes U) + (1_Q \rtimes P) )(\delta^l_{Q,U,P})^{-1}
      \\&= \delta^l_{Q,U,P}(1_{QU} + 1_{QP})(\delta^l_{Q,U,P})^{-1}
      = 1_{Q(U+P)}
    \end{align*}
    and the claim follows.
  \item By induction on $R$. The base case is $\sigma^+_{P,Q} \rtimes 0 = 1_0 = \sigma^+_{0,0} = \sigma^+_{P0,Q0}$. For the inductive case, suppose $\sigma^+_{P,Q} \rtimes R = \sigma^+_{PR,QR}$. Then using Lemma~\ref{lem:monomial-whiskering-properties} we have:
    \begin{align*}
      & \sigma^+_{P,Q} \rtimes (W + R)
      = \delta^l_{P+Q,W,R}((\sigma^+_{P,Q} \rtimes W) + (\sigma^+_{P,Q} \rtimes R))(\delta^l_{Q+P,W,R})^{-1}
      \\&= (\delta^l_{P,W,R} + \delta^l_{Q,W,R})(1_{PW} + \sigma^+_{PR,QW} + 1_{QR})(\sigma^+_{PW,QW} + \sigma^+_{PR,QR})\\&\hspace{0.5cm}(1_{QW} + \sigma^+_{PW,QR} + 1_{PW})((\delta^l_{Q,W,R})^{-1} + (\delta^l_{P,W,R})^{-1})
      \\&= (\delta^l_{P,W,R} + \delta^l_{Q,W,R})\sigma^+_{PW+PR,QW+QR}((\delta^l_{Q,W,R})^{-1} + (\delta^l_{P,W,R})^{-1})
      \\&= \sigma^+_{P(W+R),Q(W+R)}
    \end{align*}
    and the claim follows.
  \item By induction on $P$. The base case is $\eta_Q \rtimes 0 = 1_0 = \eta_0 = \eta_{Q0}$. For the inductive case, suppose $\eta_Q \rtimes P = \eta_{QP}$. Then using Lemma~\ref{lem:monomial-whiskering-properties} we have:
    \begin{align*}
      & \eta_Q \rtimes (U+P)
      = \delta^l_{0,U,P}((\eta_Q \rtimes U) + (\eta_Q \rtimes P))(\delta^l_{Q,U,P})^{-1}
      \\&= 1_0(\eta_{QU} + \eta_{QP})(\delta^l_{Q,U,P})^{-1}
      = \eta_{QU + QP}(\delta^l_{Q,U,P})^{-1}
      = \eta_{Q(U+P)}
    \end{align*}
    and the claim follows.
  \item By induction on $P$. The base case is $\mu_Q \rtimes 0 = 1_0 = \mu_0 = \mu_{Q0}$. For the inductive case, suppose $\mu_Q \rtimes P = \mu_{QP}$. Then using Lemma~\ref{lem:monomial-whiskering-properties} we have:
    \begin{align*}
      & \mu_Q \rtimes (U+P)
      = \delta^l_{Q+Q,U,P}((\mu_Q \rtimes U) + (\mu_Q \rtimes P))(\delta^l_{Q,U,P})^{-1}
      \\&= (\delta^l_{Q,U,P} + \delta^l_{Q,U,P})(1_{QU} + \sigma^+_{QP,QU} + 1_{QP})(\mu_{QU} + \mu_{QP})(\delta^l_{Q,U,P})^{-1}
      \\&= (\delta^l_{Q,U,P} + \delta^l_{Q,U,P})\mu_{QU+QP}(\delta^l_{Q,U,P})^{-1}
      = \mu_{Q(U+P)}\delta^l_{Q,U,P}(\delta^l_{Q,U,P})^{-1}
      = \mu_{Q(U+P)}
    \end{align*}
    and the claim follows.
  \item By induction on $P$. The base is $\mathsf{Tr}^Q_{R,S}(f) \rtimes 0 = 1_0 \mathsf{Tr}^0_{0,0}(1_0) = \mathsf{Tr}^{Q0}_{R0,S0}(f \rtimes 0)$. For the inductive case, suppose $\mathsf{Tr}^Q_{R,S}(f) \rtimes P = \mathsf{Tr}^{QP}_{RP,SP}(f \rtimes P)$. Then using Lemma~\ref{lem:monomial-whiskering-properties} we have:
    \begin{align*}
      & \mathsf{Tr}^Q_{R,S}(f) \rtimes (U + P)
      = \delta^l_{R,U,P}((\mathsf{Tr}^Q_{R,S}(f) \rtimes U) + (\mathsf{Tr}^Q_{R,S}(f) \rtimes P))(\delta^l_{S,U,P})^{-1}
      \\&= \delta^l_{R,U,P}(\mathsf{Tr}^{QU}_{RU,SU}(f \rtimes U) + \mathsf{Tr}^{QP}_{RP,SP}(f \rtimes P))(\delta^l_{S,U,P})^{-1}
      \\&\stackrel{*}{=} \mathsf{Tr}^{Q(U+P)}_{R(U+P),S(U+P)}( (\delta^l_{R,U,P} + \delta^l_{Q,U,P})(1_{RU} + \sigma^+_{RP,QU} + 1_{QP})((f \rtimes U) + (f \rtimes P))\\&\hspace{0.5cm}(1_{SU} + \sigma^+_{QU,SP} + 1_{QP})( (\delta^l_{S,U,P})^{-1} + (\delta^l_{Q,U,P})^{-1} ))
      \\&= \mathsf{Tr}^{Q(U+P)}_{R(U+P),S(U+P)}(\delta^l_{R+Q,U,P}((f \rtimes U) + (f \rtimes P))(\delta^l_{S+Q,U,P})^{-1})
      \\&= \mathsf{Tr}^{Q(U+P)}_{R(U+P),S(U+P)}(f \rtimes (U+P))
    \end{align*}
    where the marked equation ($\stackrel{*}{=}$) holds as in:
    \[
    \includegraphics[height=5cm]{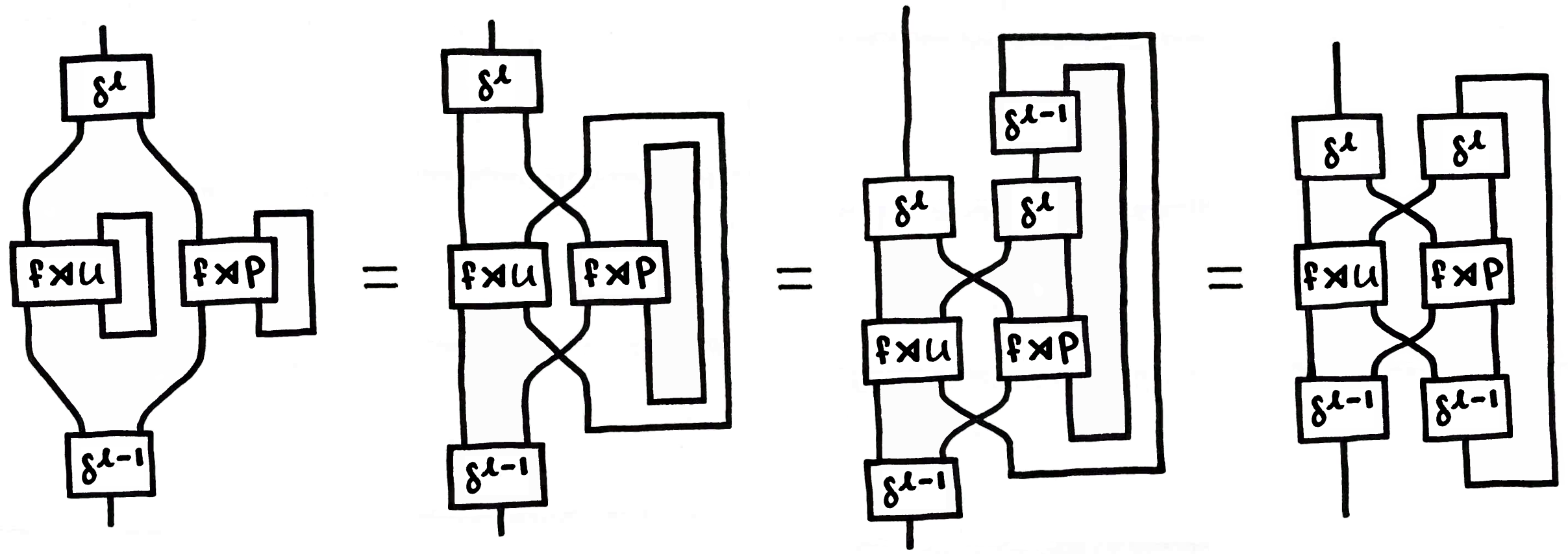}
    \]
    The claim follows. 
  \item By induction on $P$. The base case is $1_0 \rtimes 0 = 1_0$. For the inductive case, suppose $1_0 \rtimes P = 1_0$. Then we have:
    \begin{align*}
      & 1_0 \rtimes (U+P)
      = \delta^l_{0,U,P}((1_0 \rtimes U) + (1_0 \rtimes P))(\delta^l_{0,U,P})^{-1}
      = 1_0(1_0 + 1_0)1_0
      = 1_0
    \end{align*}
    and the claim follows.
      \item By induction on $P$. The base case is $ts \rtimes 0 = 1_0 = 1_01_0 = (t \rtimes 0)(s \rtimes 0)$. For the inductive case, suppose $ts \rtimes P = (t \rtimes P)(s \rtimes P)$. Then we have:
    \begin{align*}
      & ts \rtimes (U+P)
      = \delta^l_{R,U,P}((ts \rtimes U) + (ts \rtimes P))(\delta^l_{T,U,P})^{-1}
      \\&= \delta^l_{R,U,P}((t \rtimes U)(s \rtimes U) + (t \rtimes P)(s \rtimes P))(\delta^l_{T,U,P})^{-1}
      \\&= \delta^l_{R,U,P}((t \rtimes U) + (t \rtimes P))(\delta^l_{S,U,P})^{-1}\delta^l_{S,U,P}((s \rtimes U) + (s \rtimes P))(\delta^l_{T,U,P})^{-1}
      \\&= (t \rtimes (U+P))(s \rtimes (U+P))
    \end{align*}
    where $t : R \to S$ and $s : S \to T$. The claim follows.
  \item By induction on $P$. The base case is $(t + s) \rtimes 0 = 1_0 = 1_0 + 1_0 = (t \rtimes 0) + (s \rtimes 0)$. For the inductive case, suppose $(t + s) \rtimes P = (t \rtimes P) + (s \rtimes P)$. Then we have:
    \begin{align*}
      & (t + s) \rtimes (U + P)
      = \delta^l_{P_1 + P_2,U,P}(((t + s) \rtimes U) + ((t + s) \rtimes P))(\delta^l_{Q_1 + Q_2,U,P})^{-1}
      \\&= (\delta^l_{P_1,U,P} + \delta^l_{P_2,U,P})(1_{P_1U} + \sigma^+_{P_1P,P_2U} + 1_{P_2P})((t \rtimes U) + (s \rtimes Y) + (t \rtimes P) + (s \rtimes P))\\&\hspace{0.5cm}(1_{Q_1U} + \sigma^+_{Q_2U,Q_1P} + 1_{Q_2P})((\delta^l_{Q_1,U,P})^{-1} + (\delta^l_{Q_2,U,P})^{-1})
      \\&= \delta^l_{P_1,U,P}((t \rtimes U) + (t \rtimes P))(\delta^l_{Q_1,U,P})^{-1} + \delta^l_{P_2,U,P}((s \rtimes U) + (s \rtimes P))(\delta^l_{Q_2,U,P})^{-1}
      \\&= (t \rtimes (U+P)) + (s \rtimes (U + P))
    \end{align*}
    and the claim follows.
  \end{enumerate}
\end{proof}

Next, we show that the distributors behave as they ought to:
\begin{lemma}\label{lem:left-distributor-properties}
  We have:
  \begin{enumerate}
  \item $\delta^l_{0,P,Q} = 1_0$
  \item $\delta^l_{P,0,Q} = 1_{PQ}$
  \item $\delta^l_{P,Q,0} = 1_{PQ}$
  \item $\delta^l_{I,P,Q} = 1_{P + Q}$
  \item $\delta^l_{P+Q,R,S} = (\delta^l_{P,R,S} + \delta^l_{Q,R,S})(1_{PR} + \sigma^+_{PS,QR} + 1_{QS})$
  \item $\delta^l_{P,Q+R,S}(\delta^l_{P,Q,R} + 1_{PS}) = \delta^l_{P,Q,R+S}(1_{PQ} + \delta^l_{P,R,S})$
  \item $\delta^l_{PQ,R,S} = (P \ltimes \delta^l_{Q,R,S})\delta^l_{P,QR,QS}$
  \item $\delta^l_{P,QS,RS} = \delta^l_{P,Q,R} \rtimes S$
  \item $\delta^l_{P,Q,R}((P \ltimes f) + (P \ltimes g)) = (P \ltimes f + g)\delta^l_{P,Q',R'}$ for $f : Q \to Q', g : R \to R'$
  \end{enumerate}
\end{lemma}
\begin{proof}
  \begin{enumerate}
  \item Immediate.
  \item By induction on $P$. The base case is $\delta^l_{0,0,Q} = 1_0 = 1_{0Q}$. For the inductive case, suppose $\delta^l_{P,0,Q} = 1_{PQ}$. Then we have:
    \begin{align*}
      &\delta^l_{U + P,0,Q}
      = (1_{U0 + UQ} + \delta^l_{P,0,Q})(1_{U0} + \sigma^+_{UQ,P0} + 1_{PQ}) 
      \\&= (1_{UQ} + 1_{PQ})(1_{UQ} + 1_{PQ})
      = 1_{UQ + PQ}
      = 1_{(U + P)Q}
    \end{align*}
    and the claim follows.
  \item By induction on $P$. The base case is $\delta^l_{0,Q,0} = 1_0 = 1_{0Q}$. For the inductive case, suppose $\delta^l_{P,Q,0} = 1_{PQ}$. Then we have:
    \begin{align*}
      & \delta^l_{U+P,Q,0}
      = (1_{UQ + U0} + \delta^l_{P,Q,0})(1_{UQ} + \sigma^+_{U0,PQ} + 1_{P0})
      \\ &= (1_{UQ} + 1_{PQ})(1_{UQ} + 1_{PQ})
      = 1_{UQ + PQ}
      = 1_{(U+P)Q}
    \end{align*}
    and the claim follows.
  \item $I$ is a monomial so we have $\delta^l_{I,P,Q} = 1_{IP + IQ} = 1_{P + Q}$ as required.
  \item By induction on $P$. The base cases is $\delta^l_{0+Q,R,S} = \delta^l_{Q,R,S} = (1_0 + \delta^l_{Q,R,S})(1_0 + \sigma^+_{0,QR} + 1_{QS}) = (\delta^l_{0,R,S} + \delta^l_{Q,R,S})(1_{0R} + \sigma^+_{0S,QR} + 1_{QS})$. For the inductive case, suppose $\delta^l_{P+Q,R,S} = (\delta^l_{P,R,S} + \delta^l_{Q,R,S})(1_{PR} + \sigma^+_{PS,QR} + 1_{QS})$. Then we have:
    \begin{align*}
      & \delta^l_{(U+P)+Q,R,S}
      = \delta^l_{U+(P+Q),R,S}
      = (1_{UR+US} + \delta^l_{P+Q,R,S})(1_{UR} + \sigma^+_{US,PR+QR} + 1_{PS+QS})
      \\&= (1_{UR+US} + (\delta^l_{P,R,S} + \delta^l_{Q,R,S})(1_{PR} + \sigma^+_{PS,QR} + 1_{QS}))(1_{UR} + \sigma^+_{US,PR+QR} + 1_{PS+QS})
      \\&= (1_{UR+US} + \delta^l_{P,R,S} + \delta^l_{Q,R,S})(1_{UR} + \sigma^+_{US,PR} + \sigma^+_{PS,QR} + 1_{QS})(1_{UR+PR} + \sigma^+_{US,QR} + 1_{PS+QS})
      \\&= ((1_{UR+US} + \delta^l_{P,R,S})(1_{UR} + \sigma^+_{US,PR} + 1_{PS}) + \delta^l_{Q,R,S})(1_{UR+PR} + \sigma^+_{US+PS,QR} + 1_{QS})
      \\&= (\delta^l_{U+P,R,S} + \delta^l_{Q,R,S})(1_{(U+P)R}+ \sigma^+_{(U+P)S,QR} + 1_{QS})
    \end{align*}
    and the claim follows.
  \item By induction on $P$. The base case is $\delta^l_{0,Q+R,S}(\delta^l_{0,Q,R} + 1_{0S}) = 1_0 = \delta^l_{0,Q,R+S}(1_{0Q} + \delta^l_{0,R,S})$. For the inductive case, suppose $\delta^l_{P,Q+R,S}(\delta^l_{P,Q,R} + 1_{PS}) = \delta^l_{P,Q,R+S}(1_{PQ} + \delta^l_{P,R,S})$. Then we have:
    \begin{align*}
      & \delta^l_{U+P,Q+R,S}(\delta^l_{U+P,Q,R} + 1_{U+P}S)
      \\&= (1_{UQ+UR+US} + \delta^l_{P,Q+R,S})(1_{UQ+UR} + \sigma^+_{US,P(Q+R)} + 1_{PS})
      \\&\hspace{0.5cm}(1_{UQ+UR} + \delta^l_{P,Q,R} + 1_{US+PS})(1_{UQ} + \sigma^+_{UR,PQ} + 1_{PR+US+PS})
      \\&= (1_{UQ+UR+US}+\delta^l_{P,Q+R,S})(1_{UQ+UR+US}+\delta^l_{P,Q,R}+1_{PS})
      \\&\hspace{0.5cm}(1_{UQ+UR} + \sigma^+_{US,PQ+PR} + 1_{PS})(1_{UQ} + \sigma^+_{UR,PQ} + 1_{PR+US+PS})
      \\&= (1_{UQ+UR+US} + \delta^l_{P,Q+R,S}(\delta^l_{P,Q,R} + 1_{PS}))(1_{UQ+UR} + \sigma^+_{US,PQ} + 1_{PR,PS})
      \\&\hspace{0.5cm}(1_{UQ} + \sigma^+_{UR,PQ} + \sigma^+_{US,PR} + 1_{PS})
      \\&= (1_{UQ+UR+US} + \delta^l_{P,Q,R+S}(1_{PQ} + \delta^l_{P,R,S}))(1_{UQ+UR} + \sigma^+_{US,PQ} + 1_{PR,PS})
      \\&\hspace{0.5cm}(1_{UQ} + \sigma^+_{UR,PQ} + \sigma^+_{US,PR} + 1_{PS})
      \\&= (1_{UQ+UR+US} + \delta^l_{P,Q,R+S})(1_{UQ} + \sigma^+_{UR+US,PQ} + 1_{P(R+S)})
      \\&\hspace{0.5cm}(1_{UQ+PQ+UR+US} + \delta^l_{P,R,S})(1_{UQ+PQ+UR} + \sigma^+_{US,PR} + 1_{PS})
      \\&= \delta^l_{U+P,Q,R+S}(1_{(U+P)Q} + \delta^l_{U+P,R,S})
    \end{align*}
  \item First, we show that for all monomials $U$ we have $\delta^l_{UQ,R,S} = U \ltimes \delta^l_{Q,R,S}$. We do this by induction on $Q$. For the base case we have $\delta^l_{U0,R,S} = 1_0 = U \ltimes 1_0 = U \ltimes \delta^l_{0,R,S}$. For the inductive case, suppose $\delta^l_{UQ,R,S} = U \ltimes \delta^l_{Q,R,S}$. Then we have:
    \begin{align*}
      & \delta^l_{U(V+Q),R,S}
      = \delta^l_{UV+UQ,R,S}
      = (1_{UVR + UVS} + \delta^l_{UQ,R,S})(1_{UVR} + \sigma^+_{UVS,UQR} + 1_{UQS})
      \\&= ((U \ltimes 1_{VR+VS}) + (U \ltimes \delta^l_{Q,R,S}))(U \ltimes (1_{VR} + \sigma^+_{VS,QR} + 1_{RS}))
      = U \ltimes \delta^l_{V+Q,R,S}
    \end{align*}
    as required. We proceed to prove the main claim by induction on $P$. The base case is $\delta^l_{0Q,R,S} = 1_0 = (0 \ltimes \delta^l_{Q,R,S})\delta^l_{0,QR,QS}$. For the inductive case, suppose $\delta^l_{PQ,R,S} = (P \ltimes \delta^l_{Q,R,S})\delta^l_{P,QR,QS}$. Then we have:
    \begin{align*}
      & ((U+P) \ltimes \delta^l_{Q,R,S})\delta^l_{U+P,QR,QS}
      \\&= ((U \ltimes \delta^l_{Q,R,S}) + (P \ltimes \delta^l_{Q,R,S}))(1_{UQR+UQS} + \delta^l_{P,QR,QS})(1_{UQR} + \sigma^+_{UQS,PQR} + 1_{PQS})
      \\&= (\delta^l_{UQ,R,S} + \delta^l_{PQ,R,S})(1_{UQR} + \sigma^+_{UQS,PQR} + 1_{PQS})
      = \delta^l_{UQ+PQ,R,S} = \delta^l_{(U+P)Q,R,S}
    \end{align*}
    and the claim follows.
  \item By induction on $P$. The base case is $\delta^l_{0,QS,RS} = 1_0 = 1_0\rtimes S = \delta^l_{0,P,Q} \rtimes S$ For the inductive case, suppose $\delta_{P,Q,R} \rtimes S = \delta_{P,QS,RS}$. Then we have:
    \begin{align*}
      & \delta^l_{U+P,QS,RS}
      = (1_{UQS + URS} + \delta^l_{P,QS,RS})(1_{UQS} + \sigma^+_{URS,PQS} + 1_{PRS})
      \\&= (1_{UQS+URS} + (\delta^l_{P,Q,R} \rtimes S))(1_{UQS} + (\sigma^+_{UR,PQ} \rtimes S) + 1_{PRS})
      \\&= (1_{UQ+UR} + \delta^l_{P,Q,R})(1_{UQ} + \sigma^+_{UR,PQ} + 1_{PR}) \rtimes S
      = \delta^l_{U+P,Q,R} \rtimes S
    \end{align*}
    and the claim follows.
  \item By induction on $P$. For the base case we have $\delta^l_{0,Q,R}((0 \ltimes f) + (0 \ltimes g)) = 1_0 = (0 \ltimes f + g)\delta^l_{P,Q',R'}$. For the inductive case, suppose $\delta^l_{P,Q,R}((P \ltimes f) + (P \ltimes g)) = (P \ltimes f + g)\delta^l_{P,Q',R'}$. Then we have:
    \begin{align*}
      & \delta^l_{U+P,Q,R}( ((U+P) \ltimes f) + ((U + P) \ltimes g) )
      \\&= (\delta^l_{U,Q,R} + \delta^l_{P,Q,R})(1_{UQ} + \sigma^+_{UR,PQ} + 1_{PR})((U \ltimes f) + (P \ltimes f) + (U \ltimes g) + (P \ltimes g))
      \\&= (\delta^l_{U,Q,R} + \delta^l_{P,Q,R})((U \ltimes f) + (U \ltimes g) + (P \ltimes f) + (P \ltimes g))(1_{UQ'} + \sigma^+_{UR',PQ'} + 1_{PR'})
      \\&= (\delta^l_{U,Q,R}((U \ltimes f) + (U \ltimes g)) + \delta^l_{P,Q,R}((P \ltimes f) + (P \ltimes g)))(1_{UQ'} + \sigma^+_{UR',PQ'} + 1_{PR'})
      \\&= ((U \ltimes f + g)\delta^l_{U,Q',R'} + (P \ltimes f + g)\delta^l_{P,Q',R'})(1_{UQ'} + \sigma^+_{UR',PQ'} + 1_{PR'})
      \\&= ((U \ltimes f + g) + (P \ltimes f + g))\delta^l_{U+P,Q,R}
      = ((U + P) \ltimes f + g)\delta^l_{U+P,Q,R}
    \end{align*}
    and the claim follows.
  \end{enumerate}
\end{proof}

We establish a few more properties of the whiskering functors:
\begin{lemma}
  For any morphism $f : R \to S$ of $\A$ and any monomials $U,V$ we have:
  \begin{enumerate}
  \item $UV \ltimes f = U \ltimes (V \ltimes f)$
  \item $f \rtimes UV = (f \rtimes U) \rtimes V$
  \item $(U \ltimes f) \rtimes V = U \ltimes (f \rtimes V)$
  \end{enumerate}
\end{lemma}
\begin{proof}
  Straightforward, by induction on $f$. 
\end{proof}

\begin{lemma}\label{lem:polynomial-coherence}
  For any morphism $f : R \to S$ of $\A$ and any polynomials $P,Q$, we have:
  \begin{enumerate}
  \item $I \ltimes f = f$
  \item $f \rtimes I = f$
  \item $PQ \ltimes f = P \ltimes (Q \ltimes f)$
  \item $f \rtimes PQ = (f \rtimes P) \rtimes Q$
  \item $(P \ltimes f) \rtimes Q = P \ltimes (f \rtimes Q)$
  \end{enumerate}
\end{lemma}
\begin{proof}
  \begin{enumerate}
  \item By induction on $f$. We use in particular that $I$ is monomial. The base cases are as follows:
    \begin{itemize}
    \item ($\s$): If $f$ is $\s_{U,V}$ then we have $I \ltimes \s_{U,V} = \s_{IU,V} = \s_{U,V}$.

    \item ($\z$): If $f$ is $\z_{U,V}$ then we have $I \ltimes \z_{U,V} = \z_{IU,V} = \z_{U,V}$.

    \item ($\p$): If $f$ is $\p_{U,V}$ then we have $I \ltimes \p_{U,V} = \p_{IU,V} = \p_{U,V}$.

    \item ($1$): If $f$ is $1_U$ then we have $I \ltimes 1_U = 1_{IU} = 1_{U}$.

    \item ($\eta$): If $f$ is $\eta_U$ then we have $I \ltimes \eta_U = \eta_{IU} = \eta_{U}$.
      
    \item ($\mu$): If $f$ is $\mu_U$ then we have $I \ltimes \mu_U = \mu_{IU} = \mu_U$.

    \item ($0$): If $f$ is $1_0$ then we have $I \ltimes 1_0 = 1_0$.

    \item ($\sigma$): If $f$ is $\sigma^+_{U,V}$ then we have $I \ltimes \sigma^+_{U,V} = \sigma^+_{IU,IV} = \sigma^+_{U,V}$. 
    \end{itemize}
    For the inductive cases, suppose $I \ltimes f = f$ and $I \ltimes g = g$. Then we have:
   \begin{itemize}
     \item ($\circ$): $I \ltimes fg = (I \ltimes f)(I \ltimes g) = fg$

     \item ($+$): $I \ltimes (f+g) = (I \ltimes f) + (I \ltimes g) = f + g$

     \item ($\mathsf{Tr}$): $I \ltimes \mathsf{Tr}^W_{P,Q}(f) = \mathsf{Tr}^{IW}_{IP,IQ}(I \ltimes f) = \mathsf{Tr}^{W}_{P,Q}(f)$
   \end{itemize}
   The claim follows.
 \item Similar to (i).
 \item We begin by showing that for any monomial $U$ we have $U \ltimes (Q \ltimes f) = UQ \ltimes f$, which we do by induction on $Q$. The base case is $U \ltimes (0 \ltimes f) = 1_0 = U0 \ltimes f$. For the inductive case suppose $U \ltimes (Q \ltimes f) = UQ \ltimes f$. Then we have:
   \begin{align*}
     & U \ltimes ((V + Q) \ltimes f)
     = U \ltimes ((V \ltimes f) + (Q \ltimes f))
     = (U \ltimes (V \ltimes f)) + (U \ltimes (Q \ltimes f))
     \\&= (UV \ltimes f) + (UQ \ltimes f)
     = (UV + UQ)\ltimes f
     = U(V+Q) \ltimes f
   \end{align*}
   as required. We are now ready to prove the primary claim, which we do by induction on $P$. The base case is $0 \ltimes (Q \ltimes f) = 1_0 = 0Q \ltimes f$. For the inductive case suppose $P \ltimes (Q \ltimes f) = PQ \ltimes f$. Then we have:
   \begin{align*}
     & (U+P) \ltimes (Q \ltimes f)
     = (U \ltimes (Q \ltimes f)) + (P \ltimes (Q \ltimes f))
     = (UQ \ltimes f) + (UP \ltimes f)
     \\&= (UQ + PQ) \ltimes f
     = (U+P)Q \ltimes f
   \end{align*}
   The claim follows. 
 \item We begin by showing that for any monomial $U$ we have $(f \rtimes U) \rtimes Q = f \rtimes UQ$, which we do by induction on $Q$. The base case is $(f \rtimes U) \rtimes 0 = 1_0 = f \rtimes U0$. For the inductive case suppose $(f \rtimes U) \rtimes Q = f \rtimes UQ$. Then we have:
   \begin{align*}
     & (f \rtimes U) \rtimes (V+Q)
     = \delta^l_{RU,V,Q}(((f \rtimes U) \rtimes V) + ((f \rtimes U) \rtimes Q))(\delta^l_{SU,V,Q})^{-1}
     \\&= (R \ltimes \delta^l_{U,V,Q})\delta^l_{R,UV,UQ}((f \rtimes UV) + (f \rtimes UQ))(\delta^l_{S,UV,UQ})^{-1}(S \ltimes (\delta^l_{U,V,Q})^{-1})
     \\&= (R \ltimes 1_{UV+UQ})(f \rtimes (UV+UQ))(S \ltimes 1_{UV+UQ})
     = f \rtimes U(V+Q)
   \end{align*}
   as required. We are now ready to prove the primary claim, which we do by induction on $P$. The base case is $(f \rtimes 0) \rtimes Q = 1_0 = f \rtimes 0Q$. For the inductive case suppose $(f \rtimes P) \rtimes Q = f \rtimes PQ$. Then we have:
   \begin{align*}
     & (f \rtimes (U+P)) \rtimes Q
     = (\delta^l_{R,U,P}((f \rtimes U) + (f \rtimes P))(\delta^l_{S,U,P})^{-1}) \rtimes Q
     \\&= (\delta^l_{R,U,P} \rtimes Q)(((f \rtimes U) \rtimes Q) + ((f \rtimes P) \rtimes Q))((\delta^l_{S,U,P})^{-1} \rtimes Q)
     \\&= \delta^l_{R,UQ,PQ}((f \rtimes UQ) + (f \rtimes PQ))(\delta^l_{S,UQ,PQ})^{-1}
     \\&= f \rtimes (UQ + PQ)
     = f \rtimes (U+P)Q
   \end{align*}
   The claim follows.
 \item We begin by showing that for any monomial $V$, $(P \ltimes f) \rtimes V = P \ltimes (f \rtimes V)$, which we do by induction on $P$. The base case is $(0 \ltimes f) \rtimes V = 1_0 = 0 \ltimes (f \rtimes V)$. For the inductive case suppose $(P \ltimes f) \rtimes V = P \ltimes (f \rtimes V)$. Then we have:
   \begin{align*}
     & ((U+P) \ltimes f) \rtimes V
     = ((U \ltimes f) + (P \ltimes f)) \rtimes V
     = ((U \ltimes f) \rtimes V) + ((P \ltimes f) \rtimes V)
     \\&= (U \ltimes (f \rtimes V)) + (P \ltimes (f \rtimes V))
     = (U+P) \ltimes (f \rtimes V)
   \end{align*}
   as required. We are now ready to prove the primary claim, which we do by induction on $Q$. The base case is $(P \ltimes f) \rtimes 0 = 1_0 = P \ltimes (f \rtimes 0)$. For the inductive case suppose $(P \ltimes f) \rtimes Q = P \ltimes (f \rtimes Q)$. Then we have:
   \begin{align*}
     & (P \ltimes f) \rtimes (V+Q)
     = \delta^l_{PR,V,Q}(((P \ltimes f) \rtimes V) + ((P \ltimes f) \rtimes Q))(\delta^l_{PS,V,Q})^{-1}
     \\&= \delta^l_{PR,V,Q}((P \ltimes + (f \rtimes V)) + (P \ltimes (f \rtimes Q)))(\delta^l_{PS,V,Q})^{-1}
     \\&= \delta^l_{PR,V,Q}(\delta^l_{P,RV,RQ})^{-1}(P \ltimes ((f \rtimes V) + (f \rtimes Q)))\delta^l_{P,SV,SQ}(\delta^l_{PS,V,Q})^{-1}
     \\&= (P \ltimes \delta^l_{R,V,Q})(P \ltimes ((f \rtimes V) + (f \rtimes Q)))(P \ltimes (\delta^l_{S,V,Q})^{-1})
     \\&= P \ltimes (\delta^l_{R,V,Q}((f \rtimes V) + (f \rtimes Q))(\delta^l_{S,V,Q})^{-1})
     \\&= P \ltimes (f \rtimes (V+Q))
   \end{align*}
   The claim follows.
  \end{enumerate}
\end{proof}

At long last, we are able to show that left and right whiskering interchange in the necessary manner:
\begin{lemma}\label{lem:polynomial-interchange}
  For all $f : P \to Q$ and $g : R \to S$ of $\A$, $(f \rtimes R)(Q \ltimes g) = (P \ltimes g)(f \rtimes S)$. 
\end{lemma}
\begin{proof}
  By induction on $f$ and $g$. The base cases are split into two groups. First, in case both $f$ and $g$ are instances of $\s$,$\z$, or $\p$ we have:
  \begin{itemize}
  \item ($\s$-$\s$): If $f$ is $\s_{U,V} : UNV \to UNV$ and $g$ is $\s_{U',V'} : U'NV' \to U'NV'$ then \textbf{[N3]} gives:
    \begin{align*}
      & (\s_{U,V} \rtimes U'NV')(UNV \ltimes \s_{U',V'})
      = \s_{U,VU'NV'}\s_{UNVU',V'}
      \\&= \s_{UNVU',V'}\s_{U,VU'NV'}
      = (UNV \ltimes \s_{U',V'})(\s_{U,V} \rtimes U'NV')
    \end{align*}
    as required.
  \item ($\s$-$\z$): If $f$ is $\s_{U,V} : UNV \to UNV$ and $g$ is $\z_{U',V'} : U'V' \to U'NV'$ then \textbf{[N4]} gives:
    \begin{align*}
      & (\s_{U,V} \rtimes U'V')(UNV \ltimes \z_{U',V'})
      = \s_{U,VU'V'}\z_{UNVU',V'}
      \\&= \z_{UNVU',V'}\s_{U,VU'NV'}
      = (UNV \ltimes \z_{U',V'})(\s_{U,V} \rtimes U'V')
    \end{align*}
  \item ($\z$-$\s$): If $f$ is $\z_{U,V} : UV \to UNV$ and $g$ is $\s_{U',V'} : U'NV' \to U'NV'$ then \textbf{[N5]} gives:
    \begin{align*}
      & (\z_{U,V} \rtimes U'NV')(UNV \ltimes \s_{U',V'})
      = \z_{U,VU'NV'}\s_{UNVU',V'}
      \\&= \s_{UVU',V'}\z_{U,VU'NV'}
      = (UV \ltimes \s_{U',V'})(\z_{U,V} \rtimes U'NV')
    \end{align*}
    as required.
  \item ($\z$-$\z$): If $f$ is $\z_{U,V} : UV \to UNV$ and $g$ is $\z_{U',V'} : U'V' \to U'NV'$ then \textbf{[N6]} gives:
    \begin{align*}
      & (\z_{U,V} \rtimes U'V')(UNV \ltimes \z_{U',V'})
      = \z_{U,VU'V'}\z_{UNVU',V'}
      \\&= \z_{UVU',V'}\z_{U,VU'NV'}
      = (UV \ltimes \z_{U',V'})(\z_{U,V} \rtimes U'NV')
    \end{align*}
  \item ($\s$-$\p$): If $f$ is $\s_{U,V} : UNV \to UNV$ and $g$ is $\p_{U',V'} : U'NV' \to U'V' + U'NV'$ then \textbf{[N7]} gives:
    \begin{align*}
      & (\s_{U,V} \rtimes U'NV')(UNV \ltimes \p_{U',V'})
      = \s_{U,VU'NV'}\p_{UNVU',V'}
      \\&= \p_{UNVU',V'}(\s_{U,VU'V'} + \s_{U,VU'NV'})
      \\&= (UNV \ltimes \p_{U',V'})((\s_{U,V} \rtimes U'V') + (\s_{U,V} \rtimes U'NV'))
      \\&= (UNV \ltimes \p_{U',V'})(\s_{U,V} \rtimes U'V' + U'NV')
    \end{align*}
    as required.
  \item ($\z$-$\p$): If $f$ is $\z_{U,V} : UV \to UNV$ and $g$ is $\p_{U',V'} : U'NV' \to U'V' + U'NV'$ then \textbf{[N8]} gives:
    \begin{align*}
      & (\z_{U,V} \rtimes U'NV')(UNV \ltimes \p_{U',V'})
      = \z_{U,VU'NV'}\p_{UNVU',V'}
      \\&= \p_{UVU',V'}(\z_{U,VU'V'} + \z_{U,VU'NV'})
      \\&= (UV \ltimes \p_{U',V'})( (\z_{U,V} \rtimes U'V') + (\z_{U,V} \rtimes U'NV'))
      \\&= (UV \ltimes \p_{U',V'})(\z_{U,V} \rtimes U'V' + U'NV')
    \end{align*}
    as required.
  \item ($\p$-$\s$): If $f$ is $\p_{U,V} : UNV \to UV + UNV$ and $g$ is $\s_{U',V'} : U'NV' \to U'NV'$ then $\textbf{[N9]}$ gives:
    \begin{align*}
      & (\p_{U,V} \rtimes U'NV')((UV + UNV) \ltimes \s_{U',V'})
      \\&= (\p_{U,V} \rtimes U'NV')((UV \ltimes \s_{U',V'}) + (UNV + \ltimes \s_{U',V'}))
      \\&= \p_{U,VU'NV'}(\s_{UVU',V'} + \s_{UNVU',V'})
      = \s_{UNVU',V'}\p_{U,VU'NV'}
      \\&= (UNV \ltimes \s_{U',V'})(\p_{U,V} \rtimes U'NV')
    \end{align*}
    as required.
  \item ($\p$-$\z$): If $f$ is $\p_{U,V} : UNV \to UV + UNV$ and $g$ is $\z_{U',V'} : U'V' \to U'NV'$ then $\textbf{[N10]}$ gives:
    \begin{align*}
      & (\p_{U,V} \rtimes U'V')((UV + UNV) \ltimes \z_{U',V'})
      = \p_{U,VU'V'}(\z_{UVU',V'} + \z_{UNVU',V'})
      \\&= \z_{UNVU',V'}\p_{U,VU'NV'}
      = (UNV \ltimes \z_{U',V'})(\p_{U,V} \rtimes U'NV')
    \end{align*}
    as required.
  \item ($\p$-$\p$): If $f$ is $\p_{U,V} : UNV \to UV + UNV$ and $g$ is $p_{U',V'} : U'NV' \to U'V' + U'NV'$ then $\textbf{[N11]}$ gives:
    \begin{align*}
      & (\p_{U,V} \rtimes U'NV')((UV + UNV) \ltimes \p_{U',V'})
      \\&= \p_{U,VU'NV'}((UV \ltimes \p_{U',V'}) + (UNV \ltimes \p_{U',V'}))
      \\&= \p_{U,VU'NV'}(\p_{UVU',V'} + \p_{UNVU',V'})
      \\&= \p_{UNVU',V'}(\p_{U,VU'V'} + \p_{U,VU'NV'})(1_{UVU'V'} + \sigma^+_{UVU'NV',UNVU'V'} + 1_{UNVU'NV'})
      \\&= (UNV \ltimes \p_{U',V'})\delta^l_{UNV,U'V',U'NV'}((\p_{U,V} \rtimes U'V') + (\p_{U,V} \rtimes U'NV'))(\delta^l_{UV+UNV,U'V',U'NV'})^{-1}
      \\&= (UNV \ltimes \p_{U',V'})(\p_{U,V} \rtimes U'V' + U'NV')
    \end{align*}
    as required. 
  \end{itemize}
  Next, in case neither $f$ nor $g$ is an instance of $\s$,$\z$, or $\p$ we have:
  \begin{itemize}
  \item ($1$-$g$): If $f$ is $1_U : U \to U$ and $g : R \to S$ then we have:
    \begin{align*}
      & (1_U \rtimes R)(U \ltimes g)
      = 1_{UR}(U \ltimes g)
      = (U \ltimes g)1_{US}
      = (U \ltimes g)(1_U \rtimes S)
    \end{align*}
    as required.
  \item ($f$-$1$): If $f : R \to S$ and $g$ is $1_U : U \to U$ then we have:
    \begin{align*}
      & (f \rtimes U)(S \ltimes 1_U)
      = (f \rtimes U)1_{SU}
      = 1_{RU}(f \rtimes U)
      = (R \ltimes 1_U)(f \rtimes U)
    \end{align*}
    as required.
  \item ($\eta$-$g$): If $f$ is $\eta_U : 0 \to U$ and $g : R \to S$ then we have:
    \begin{align*}
      & (\eta_U \rtimes R)(U \ltimes g)
      = \eta_{UR}(U \ltimes g)
      = \eta_{US}
      = 1_0\eta_{US}
      = (0 \ltimes g)(\eta_U \rtimes S)
    \end{align*}
    as required.
  \item ($f$-$\eta$): If $f : R \to S$ and $g$ is $\eta_U : 0 \to U$ then we have:
    \begin{align*}
      & (f \rtimes 0)(S \ltimes \eta_U)
      = 1_0\eta_{SU}
      = \eta_{SU}
      = \eta_{RU}(f \rtimes U)
      = (R \ltimes \eta_U)(f \rtimes U)
    \end{align*}
    as required.
  \item ($\mu$-$g$): If $f$ is $\mu_U : U + U \to U$ and $g : R \to S$ then we have:
    \begin{align*}
      & (\mu_U \rtimes R)(U \ltimes g)
      = \mu_{UR}(U \ltimes g)
      = ((U \ltimes g) + (U \ltimes g))\mu_{US}
      = ((U+U) \ltimes g)(\mu_U \rtimes S)
    \end{align*}
    as required.
  \item ($f$-$\mu$): If $f : R \to S$ and $g$ is $\mu_U : U+U \to U$ then we have:
    \begin{align*}
      & (f \rtimes (U+U))(S \ltimes \mu_U)
      = \delta^l_{R,U,U}((f \rtimes U) + (f \rtimes U))(\delta^l_{S,U,U})^{-1}\delta^l_{S,U,U}\mu_{SU}
      \\&= \delta^l_{R,U,U}\mu_{RU}(f \rtimes U)
      = (R \ltimes \mu_U)(f \rtimes U)
    \end{align*}
    as required.
  \item ($0$-$g$): If $f$ is $1_0 : 0 \to 0$ and $g : R \to S$ then we have:
    \begin{align*}
      & (1_0 \rtimes R)(0 \ltimes g)
      = 1_0
      = (0 \ltimes g)(1_0 \rtimes S)
    \end{align*}
    as required.
  \item ($f$-$0$): If $f : R \to S$ and $g$ is $1_0 : 0 \to 0$ then we have:
    \begin{align*}
      & (f \rtimes 0)(S \ltimes 1_0)
      = 1_0
      = (R \ltimes 1_0)(f \rtimes 0)
    \end{align*}
    as required.
  \item ($\sigma$-$g$): If $f$ is $\sigma^+_{U,V} : U+V \to V+U$ and $g : R \to S$ then we have:
    \begin{align*}
      & (\sigma^+_{U,V} \rtimes R)((V+U) \ltimes g)
      = \sigma^+_{UR,VR}((V \ltimes g) + (U \ltimes g))
      \\&= ((U \ltimes g) + (V \ltimes g))\sigma^+_{US,VS}
      = ((U+V) \ltimes g)(\sigma^+_{U,V} \rtimes S)
    \end{align*}
    as required.
  \item ($f$-$\sigma$): If $f : R \to S$ and $g$ is $\sigma^+_{U,V} : U+V \to V+U$ then we have:
    \begin{align*}
      & (f \rtimes (U+V))(S \ltimes \sigma^+_{U,V})
      = \delta^l_{R,U,V}((f \rtimes U) + (f \rtimes V))(\delta^l_{S,U,V})^{-1}(S \ltimes \sigma^+_{U,V})
      \\&= \delta^l_{R,U,V}((f \rtimes U) + (f \rtimes V))\sigma^+_{SU,SV}(\delta^l_{S,V,U})^{-1}
      = \delta^l_{R,U,V}\sigma^+_{RU,RV}((f \rtimes U) + (f \rtimes V))(\delta^l_{S,V,U})^{-1}
      \\&= (R \ltimes \sigma^+_{U,V})\delta^l_{R,V,U}((f \rtimes V) + (f \rtimes U))(\delta^l_{S,V,U})^{-1}
      = (R \ltimes \sigma^+_{U,V})(f \rtimes (V+U))
    \end{align*}
    as required.
  \end{itemize}
  Finally, the inductive cases are as follows:
  \begin{itemize}
  \item ($\mathsf{Tr}$-$g$): Given $\mathsf{Tr}^U_{P,Q}(f) : P \to Q$ and $g : R \to S$ for some $f : P+U \to Q+U$ such that $(f \rtimes R)((Q+U) \ltimes g) = ((P+U) \ltimes g)(f \rtimes S)$ we have:
    \begin{align*}
      & (f \rtimes R)((Q \ltimes g) + (U \ltimes g))
      = (f \rtimes R)((Q+U) \ltimes g)
      \\&= ((P+U) \ltimes g)(f \rtimes S)
      = ((P \ltimes g) + (U \ltimes g))(f \rtimes S)
    \end{align*}
    and so by the uniformity axiom we have
    \[
    \mathsf{Tr}^{UR}_{PR,QR}(f \rtimes R)(Q \ltimes g)
    =
    (P \ltimes g)\mathsf{Tr}^{US}_{PS,QS}(f \rtimes S)
    \]
    as in:
    \[
    \includegraphics[height=3cm,align=c]{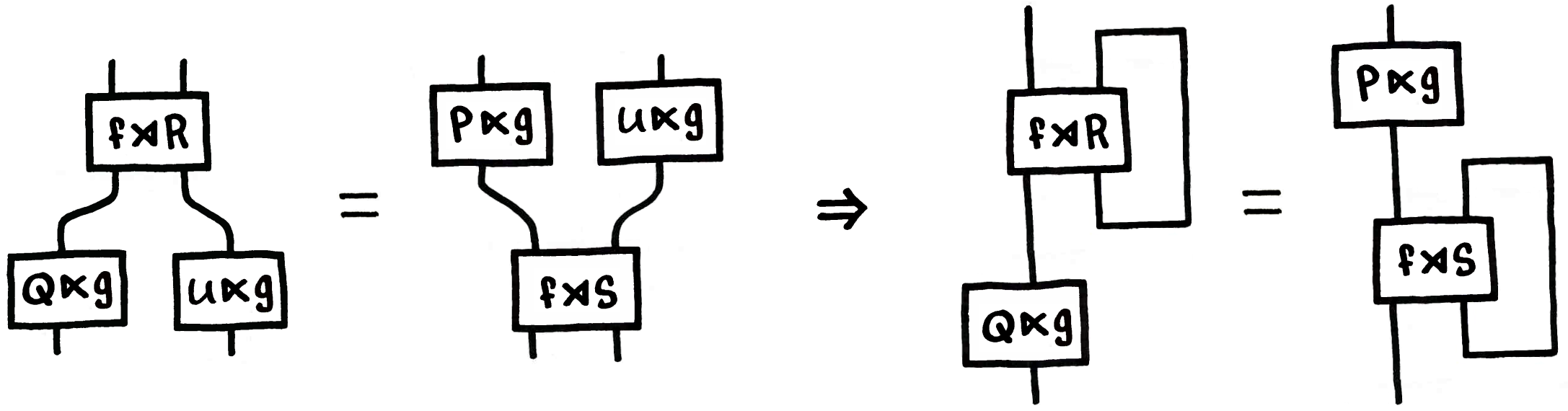}
    \]
    which in turn gives:
    \begin{align*}
      & (\mathsf{Tr}^U_{P,Q}(f) \rtimes R)(Q \ltimes g)
      = \mathsf{Tr}^{UR}_{PR,QR}(f \rtimes R)(Q \ltimes g)
      \\&= (P \ltimes g)\mathsf{Tr}^{US}_{PS,QS}(f \rtimes S)
      = (P \ltimes g)(\mathsf{Tr}^U_{P,Q}(f) \rtimes S)
    \end{align*}
    as required.
  \item ($f$-$\mathsf{Tr}$): Given $f : R \to S$ and $\mathsf{Tr}^{U}_{P,Q}(g) : P \to Q$ for some $g$ with $(f \rtimes (P+U))(S \ltimes g) = (R \ltimes g)(f \rtimes (Q+U))$ we have:
    \begin{align*}
      & (\delta^l_{R,P,U})^{-1}(R \ltimes g)\delta^l_{S,Q,U}((f \rtimes Q) + (f \rtimes U))
      = (\delta^l_{R,P,U})^{-1}(R \ltimes g)(f \rtimes (Q+U))\delta^l_{S,Q,U}
      \\&= (\delta^l_{R,P,U})^{-1}(f \rtimes (P+U))(S \ltimes g)\delta^l_{S,Q,U}
      = ((f \rtimes P) + (f \rtimes U))(\delta^l_{S,P,U})^{-1}(S \ltimes g)\delta^l_{S,Q,U}
    \end{align*}
    and so by the uniformity axiom we have:
    \[
    \mathsf{Tr}^{RU}_{RP,RQ}((\delta^l_{R,P,U})^{-1}(R \ltimes g)\delta^l_{R,Q,U})(f \rtimes Q)
    =
    (f \rtimes P)\mathsf{Tr}^{SU}_{SP,SQ}((\delta^l_{S,P,U})^{-1}(S \ltimes g)\delta^l_{S,Q,U})
    \]
    as in:
    \[
    \includegraphics[height=4.5cm,align=c]{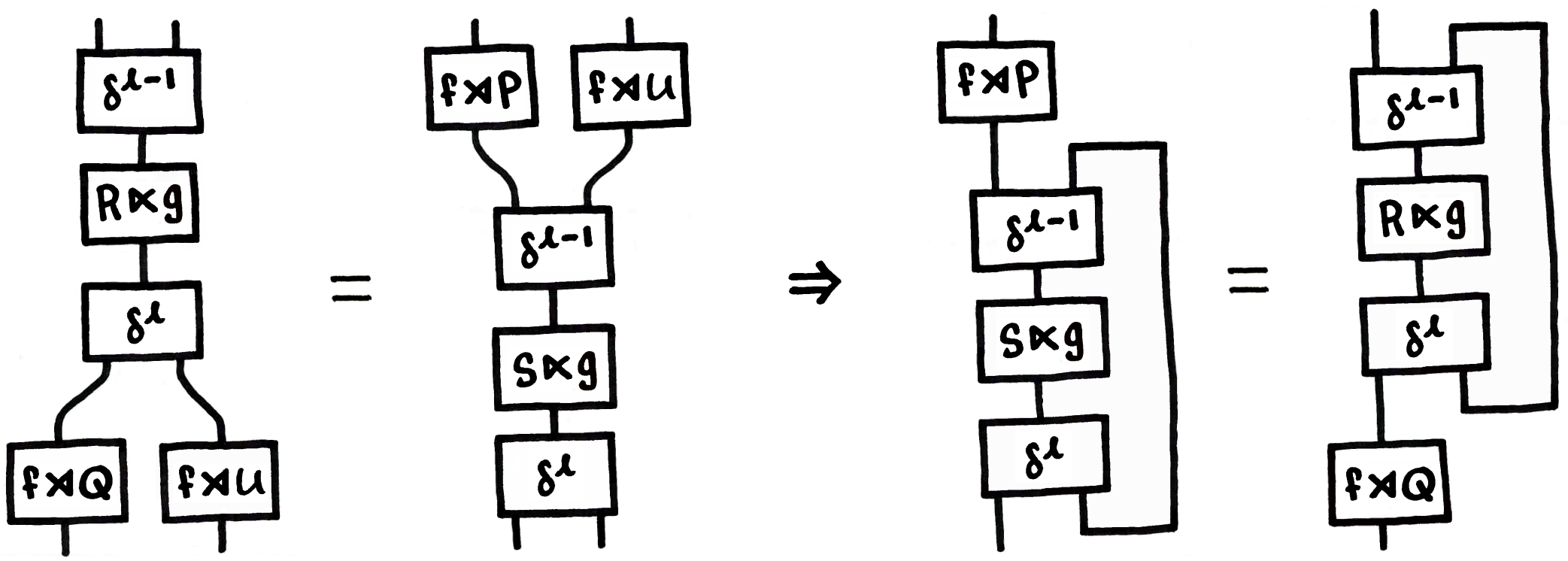}
    \]
    which in turn gives:
    \begin{align*}
      & (f \rtimes P)(S \ltimes \mathsf{Tr}^U_{P,Q}(g))
      = (f \rtimes P)\mathsf{Tr}^{SU}_{SP,SQ}((\delta^l_{S,P,U})^{-1}(S \ltimes g)\delta^l_{S,Q,U})
      \\&= \mathsf{Tr}^{RU}_{RP,RQ}((\delta^l_{R,P,U})^{-1}(R \ltimes g)\delta^l_{R,Q,U})(f \rtimes Q)
      = (R \ltimes \mathsf{Tr}^U_{P,Q}(g))(f \rtimes Q)
    \end{align*}
    as required.
  \item ($\circ$-$g$): Given $f_1 : P \to T$, $f_2 : T \to Q$, and $g : R \to S$ such that $(f_1 \rtimes R)(T \ltimes g) = (P \ltimes g)(f_1 \rtimes S)$ and $(f_2 \rtimes R)(Q \ltimes g) = (T \ltimes g)(f_2 \rtimes S)$, we have:
    \begin{align*}
      & (f_1f_2 \rtimes R)(Q \ltimes g)
      = (f_1 \rtimes R)(f_2 \rtimes R)(Q \ltimes g)
      \\&= (f_1 \rtimes R)(T \ltimes g)(f_2 \rtimes S)
      = (P \ltimes g)(f_1f_2 \rtimes S)
    \end{align*}
    as required. 
  \item ($f$-$\circ$): Given $f : R \to S$, $g_1 : P \to T$, and $g_2 : T \to Q$ such that $(f \rtimes P)(S \ltimes g_1) = (R \ltimes g_1)(f \rtimes T)$ and $(f \rtimes T)(S \ltimes g_2) = (R \ltimes g_2)(f \rtimes Q)$, we have:
    \begin{align*}
      & (f \rtimes P)(S \ltimes g_1g_2)
      = (f \rtimes P)(S \ltimes g_1)(S \ltimes g_2)
      \\&= (R \ltimes g_1)(f \rtimes T)(S \ltimes g_2)
      = (R \ltimes g_1g_2)(f \rtimes Q)
    \end{align*}
    as required. 
  \item ($+$-$g$): Given $f_1 + f_2 : P_1 + P_2 \to Q_1 + Q_2$ and $g : R \to S$ such that $(f_1 \rtimes R)(Q_1 \ltimes g) = (P_1 \ltimes g)(f_1 \rtimes S)$ and $(f_2 \rtimes R)(Q_2 \ltimes g) = (P_2 \ltimes g)(f_2 \rtimes S)$, we have:
    \begin{align*}
      & ((f_1 + f_2) \rtimes R)((Q_1 + Q_2) \ltimes g)
      = ((f_1 \rtimes R) + (f_2 \rtimes R))((Q_1 \ltimes g) + (Q_2 \ltimes g))
      \\&= (f_1 \rtimes R)(Q_1 \ltimes g) + (f_2 \rtimes R)(Q_2 \ltimes g)
      = (P_1 \ltimes g)(f_1 \rtimes S) + (P_2 \ltimes g)(f_2 \rtimes S)
      \\&= ((P_1 + P_2) \ltimes g) + ((f_1 + f_2) \rtimes S)
    \end{align*}
    as required. 
  \item ($f$-$+$): Given $f: R \to S$ and $g_1 + g_2 : P_1 + P_2 \to Q_1 + Q_2$ such that $(f \rtimes P_1)(S \ltimes g_1) = (R \ltimes g_1)(f \rtimes Q_1)$ and $(f \rtimes P_2)(S \ltimes g_2) = (R \ltimes g_2)(f \rtimes Q_2)$, we have:
    \begin{align*}
      & (f \rtimes (P_1 + P_2))(S \ltimes (g_1 + g_2))
      \\&= \delta^l_{R,P_1,P_2}((f \rtimes P_1) + (f \rtimes P_2))(\delta^l_{S,P_1,P_2})^{-1}\delta^l_{S,P_1,P_2}((S \ltimes g_1) + (S \ltimes g_2))(\delta^l_{S,Q_1,Q_2})^{-1}
      \\&= \delta^l_{R,P_1,P_2}((f \rtimes P_1)(S \ltimes g_1) + (f \rtimes P_2)(S \ltimes g_2))(\delta^l_{S,Q_1,Q_2})^{-1}
      \\&= \delta^l_{R,P_1,P_2}((R \ltimes g_1)(f \rtimes Q_1) + (R \ltimes g_2)(f \rtimes Q_2))(\delta^l_{S,Q_1,Q_2})^{-1}
      \\&= \delta^l_{R,P_1,P_2}((R \ltimes g_1) + (R \ltimes g_2))(\delta^l_{R,Q_1,Q_2})^{-1}\delta^l_{R,Q_1,Q_2}((f \rtimes Q_1) + (f \rtimes Q_2))(\delta^l_{S,Q_1,Q_2})^{-1}
      \\&= (R \ltimes (g_1 + g_2))(f \rtimes (Q_1 + Q_2))
    \end{align*}
  \end{itemize}
  and the claim follows by induction. 
\end{proof}

This allows us to prove that our proposed multiplicative monoidal structure on $\A$ does in fact define a monoidal category:
\begin{lemma}\label{lem:tensor-monoidal-structure}
  $(\A,\otimes,I)$ is a strict monoidal category with $- \otimes - : \A \times \A \to \A$ defined on arrows $f : P \to Q$ and $g : R \to S$ of $\A$ as in:
  \[
  f \otimes g = (f \rtimes R)(Q \ltimes g)
  \]
\end{lemma}
\begin{proof}
  We begin by showing $-\otimes-$ is a functor. It preserves identities as in:
  \begin{align*}
    & 1_P \otimes 1_Q
    = (1_P \rtimes Q)(P \ltimes 1_Q)
    = 1_{PQ}1_{PQ}
    = 1_{PQ}
    = 1_{P \otimes Q}
  \end{align*}
  and given $f : P \to Q$, $g : R \to S$, $h : Q \to T$, $k : S \to H$ we have that $-\otimes-$ preserves composition as in:
  \begin{align*}
    & (f \otimes g)(h \otimes k)
    = (f \rtimes R)(Q \ltimes g)(h \rtimes S)(T \ltimes k)
    \\&= (f \rtimes R)(h \rtimes R)(T \ltimes g)(T \ltimes k)
    \\&= (hf \rtimes R)(T \ltimes gk)
    = hf \otimes gk
  \end{align*}
  Moreover, we have that $-\otimes-$ is associative as in:
  \begin{align*}
    & (f \otimes g) \otimes h
    = ((f \otimes g) \rtimes T)(QS \ltimes h)
    \\&= ((f \rtimes R)(Q \ltimes g) \rtimes T)(QS \ltimes h)
    \\&= ((f \rtimes R) \rtimes T)((Q \ltimes g) \rtimes T)(Q \ltimes (S \ltimes h))
    \\&= ((f \rtimes R) \rtimes T)(Q \ltimes (g \rtimes T))(Q \ltimes (S \ltimes h))
    \\&= (f \rtimes RT)(Q \ltimes (g \rtimes T)(S \ltimes h))
    \\&= (f \rtimes RT)(Q \ltimes (g \otimes h))
    = f \otimes (g \otimes h)
  \end{align*}
  and finally we have that $-\otimes-$ is unital as in:
  \[
  f \otimes 1_I = (f \rtimes I)(Q \ltimes 1_I) = f1_Q = f
  \]
  and
  \[
  1_I \otimes f = (1_I \rtimes P)(I \ltimes f) = 1_Pf = f
  \]
  The claim follows.
\end{proof}

From here it is easy to see that $\otimes$ distributes over $+$:
\begin{lemma}
  $(\A,(+,0,\sigma^+),(\otimes,I),(\mu,\eta),(1,1),(\delta^l,1))$ is a right-strict distributive monoidal category. 
\end{lemma}
\begin{proof}
  We have already seen that $(\A,\otimes,I)$ is a strict monoidal category (Lemma~\ref{lem:tensor-monoidal-structure}), and we know that $(\A,+,0,\sigma^+)$ is a symmetric strict monoidal category. To see that $\delta^l$ is a natural transformation, suppose $f : P \to P'$, $g : Q \to Q'$, and $h : R \to R'$ in $\A$. Then we have:
  \begin{align*}
    & \delta^l_{P,Q,R}((f \otimes g) + (f \otimes h))
    = \delta^l_{P,Q,R}( (f \rtimes Q)(P' \ltimes g) + (f \rtimes R)(P' \ltimes h))
    \\&= \delta^l_{P,Q,R} ((f \rtimes Q) + (f \rtimes R))((P' \ltimes g) + (P' \ltimes h))
    \\&= \delta^l_{P,Q,R} ((f \rtimes Q) + (f \rtimes R))(\delta^l_{P',Q,R})^{-1}\delta^l_{P',Q,R}((P' \ltimes g) + (P' \ltimes h))
    \\&= (f \rtimes (Q+R))(P' \ltimes (g + h))\delta^l_{P',Q',R'}
    = (f \otimes (g + h))\delta^l_{P',Q',R'}
  \end{align*}
  and so $\delta^l$ is natural, as required. Next, for right dsitributivity we have:
  \begin{align*}
    & (f + g) \otimes h
    = ((f + g) \rtimes R)((P' + Q') \ltimes h)
    \\&= ((f \rtimes R) + (g \rtimes R))((P' \ltimes h) + (Q' \ltimes h))
    \\&= (f \rtimes R)(P' \ltimes h) + (g \rtimes R)(Q' \ltimes h)
    = (f \otimes h) + (g \otimes h)
  \end{align*}
  and for left and right annihilation we have:
  \[
  f \otimes 1_0 = (f \rtimes 0)(P' \ltimes 1_0) = 1_01_0 = 1_0
  \]
  and
  \[
  1_0 \otimes f = (1_0 \rtimes P)(0 \ltimes f) = 1_01_0 = 1_0
  \]
  Of the remaining conditions $\A$ must satisfy to be a right-strict distributive monoidal category, all but one are shown to hold in Lemma~\ref{lem:left-distributor-properties}. For the final condition we have $\sigma^+_{PR,QR} = \sigma^+_{P,Q} \rtimes R = \sigma^+_{P,Q} \otimes 1_R$ via Lemma~\ref{lem:polynomial-whiskering-properties}. The claim follows. 
\end{proof}

Now since $\textbf{[N1]}$ and $\textbf{[N2]}$ give an isomorphism $I+N \cong N$ we may conclude:
\begin{theorem}\label{thm:appendix-abacus-elgot}
  $(\A,(+,0,\sigma^+),(\otimes,I),(\mu,\eta),\mathsf{Tr},(1,1),(\delta^l,1),(N,\iota))$ is a right-strict Elgot category.
\end{theorem}\qed

\end{document}